\documentclass{amsart}

\title{On unimodular finite tensor categories}
\author{Kenichi Shimizu}
\date{}

\usepackage{amsmath,amssymb,amsthm,amscd,stmaryrd}
\usepackage{graphicx}
\usepackage[all]{xy}
\usepackage{bbm}
\usepackage{url}

\usepackage{comment}

\usepackage{amsthm}
\numberwithin{equation}{section}

\newtheorem{counter}{}[section]

\theoremstyle{definition}
\newtheorem{definition}         [counter]{Definition}

\theoremstyle{plain}
\newtheorem{lemma}              [counter]{Lemma}

\newtheorem{theorem}            [counter]{Theorem}
\newtheorem{corollary}          [counter]{Corollary}

\newtheorem*{theorem*}           {Theorem}
\theoremstyle{remark}
\newtheorem{remark}             [counter]{Remark}
\newtheorem{example}            [counter]{Example}

\newcommand{\id}{\mathrm{id}}

\newcommand{\Hom}{\mathrm{Hom}}
\newcommand{\End}{\mathrm{End}}
\newcommand{\iHom}{\underline{\Hom}}
\newcommand{\iEnd}{\underline{\End}}
\newcommand{\iId}{\underline{\id}}
\newcommand{\ieval}{\underline{\mathrm{ev}}}
\newcommand{\icomp}{\underline{\circ}}

\newcommand{\NAT}{\text{\sc Nat}}
\newcommand{\END}{\text{\sc End}}
\newcommand{\REX}{\text{\sc Rex}}
\newcommand{\LEX}{\text{\sc Lex}}

\newcommand{\eval}{{\rm ev}}
\newcommand{\coev}{{\rm coev}}
\newcommand{\dinatto}{\xrightarrow{\ .. \ }}
\newcommand{\op}{\mathsf{op}}
\newcommand{\rev}{\mathsf{rev}}
\newcommand{\env}{\mathsf{env}}
\newcommand{\unitobj}{\mathbbm{1}}

\begin{document}

\begin{abstract}
  Let $\mathcal{C}$ be a finite tensor category with simple unit object, let $\mathcal{Z}(\mathcal{C})$ denote its monoidal center, and let $L$ and $R$ be a left adjoint and a right adjoint of the forgetful functor $U: \mathcal{Z}(\mathcal{C}) \to \mathcal{C}$. We show that the following conditions are equivalent: (1) $\mathcal{C}$ is unimodular, (2) $U$ is a Frobenius functor, (3) $L$ preserves the duality, (4) $R$ preserves the duality, (5) $L(\mathbbm{1})$ is self-dual, and (6) $R(\mathbbm{1})$ is self-dual, where $\mathbbm{1} \in \mathcal{C}$ is the unit object. We also give some other equivalent conditions. As an application, we give a categorical understanding of some topological invariants arising from finite-dimensional unimodular Hopf algebras.
\end{abstract}

\maketitle

\section{Introduction}

For a locally compact group $G$ with right Haar measure $\mu$, the modular function is defined as the unique function $\alpha: G \to \mathbb{R}_{+}$ such that $\mu(g E) = \alpha(g) \mu(E)$ for all Borel subsets $E$ of $G$. If the modular function is constantly one, then $G$ is said to be unimodular. We can define the modular function (usually called the distinguished grouplike element) and the unimodularity of a Hopf algebra by using the intregral theory for Hopf algebras instead of Haar measures. They are important not only in the Hopf algebra theory ({\em e.g.}, the Radford $S^4$-formula \cite{MR0407069}), but also in their applications to topology: For example, given a finite-dimensional unimodular ribbon Hopf algebra, one can construct an invariant of closed 3-manifolds \cite{MR1413901,MR1321293}. Recently, Ishii and Masuoka \cite{MR3265394} developed a method to construct an invariant of handlebody-links from finite-dimensional unimodular Hopf algebras.

A finite tensor category \cite{MR2119143} is a class of monoidal categories including the representation category of a finite-dimensional Hopf algebra. To generalize the Radford $S^4$-formula to finite tensor categories, Etingof, Nikshych and Ostrik \cite{MR2097289} introduced the {\em distinguished invertible object} $D \in \mathcal{C}$. This object is a categorical analogue of the modular function, and therefore $\mathcal{C}$ is said to be {\em unimodular} if $D \cong \unitobj$. In this paper, we investigate the object $D$ in detail and provide some characterizations of the unimodularity of a finite tensor category by using the monoidal center. As an application, we give a categorical understanding of the above-mentioned constructions of topological invariants by generalizing such constructions to unimodular finite tensor categories.

This paper is organized as follows: In Section~\ref{sec:preliminaries}, we recall from \cite{EGNO-Lect,MR2119143,MR1712872,MR1321145} some basic results on finite tensor categories and their module categories. We warn that, unlike \cite{MR2119143}, we do not assume that the unit object of a finite tensor category is simple (see \S\ref{subsec:FTC} for our definition).  In relation to this, we note that some additional technical assumptions on finite tensor categories will be made at the beginning of each of Sections~\ref{sec:main-thm}, \ref{sec:applications-1} and \ref{sec:applications-2}.

In Section~\ref{thm:central-Hopf-monad}, we first recall from \cite{MR1712872} the notions of ends and coends. Following \cite{MR2869176,MR2342829}, the (monoidal) center $\mathcal{Z}(\mathcal{C})$ of a rigid monoidal category $\mathcal{C}$ is isomorphic to the category of modules over a certain Hopf monad on $\mathcal{C}$, which we call the {\em central Hopf monad} on $\mathcal{C}$, provided that the coend
\begin{equation}
  \label{eq:intro-Hopf-monad-Z}
  Z(V) = \int^{X \in \mathcal{C}} X^* \otimes V \otimes X
\end{equation}
exists for all $V \in \mathcal{C}$. We also show that a coend of certain type of functors, including~\eqref{eq:intro-Hopf-monad-Z}, exists in a finite tensor category. As an application, we give an alternative proof of the fact that the center of a finite tensor category is again a finite tensor category \cite{MR2119143}.

Our main theorem is proved in Section~\ref{sec:main-thm}. There is an algebra $A \in \mathcal{C} \boxtimes \mathcal{C}^{\rev}$ which plays a crucial role in the definition of the distinguished invertible object of a finite tensor category $\mathcal{C}$. We express the algebra $A$ as a coend of a certain functor, and then show its relation with the central Hopf monad on $\mathcal{C}$. Based on this observation, we see that there are equivalences $K$ and $\widetilde{K}$ such that the diagram
\begin{equation*}
  \xymatrix{
    \mathcal{Z}(\mathcal{C})
    \ar[rr]^(.2){\widetilde{K}} \ar[d]_{U}
    & & \text{(the category of $A$-bimodules in $\mathcal{C} \boxtimes \mathcal{C}^{\rev}$)}
    \ar[d]^{F_A} \\
    \mathcal{C} \ar[rr]_(.2){K}
    & & \text{(the category of right $A$-modules in $\mathcal{C} \boxtimes \mathcal{C}^{\rev}$)}
  }
\end{equation*}
commutes, where $F_A$ is the functor forgetting the left $A$-module structure. Now let $L$ and $R$ be a left and a right adjoint functor of $U$. By using the above commutative diagram, we obtain a natural isomorphism
\begin{equation}
  \label{eq:intro-L-R-D}
  R(V) \cong L(D \otimes V) \quad (V \in \mathcal{C}),
\end{equation}
where $D \in \mathcal{C}$ is the distinguished invertible object of $\mathcal{C}$ (Lemma~\ref{lem:induction-L-R-1}). Once~\eqref{eq:intro-L-R-D} is obtained, the following our main theorem is proved without much difficulty:

\begin{theorem*}[Theorem~\ref{thm:main-theorem}]
  The following assertions are equivalent:
  \begin{enumerate}
  \item $\mathcal{C}$ is unimodular.
  \item $U$ is a Frobenius functor, {\it i.e.}, $L \cong R$.
  \item $L$ preserves the left duality, {\it i.e.}, there exists a natural isomorphism $L(V^*) \cong L(V)^*$ for $V \in \mathcal{C}$, where $(-)^*$ is the left duality on $\mathcal{C}$.
  \item $R$ preserves the left duality.
  \end{enumerate}
  If, moreover, the unit object $\unitobj \in \mathcal{C}$ is a simple object, then the above conditions are equivalent to each of the following conditions:
  \begin{enumerate}
    \setcounter{enumi}{4}
  \item $L(\unitobj) \cong L(\unitobj)^*$.
  \item $R(\unitobj) \cong R(\unitobj)^*$.
  \item $\Hom_{\mathcal{Z}(\mathcal{C})}(\unitobj, L(\unitobj)) \ne 0$.
  \item $\Hom_{\mathcal{Z}(\mathcal{C})}(R(\unitobj), \unitobj) \ne 0$.
  \end{enumerate}
\end{theorem*}

We note that the equivalence between (1) and (2) has been obtained by Caenepeel, Militaru and Zhu in \cite[\S4, Theorem 53]{MR1926102} in the case where $\mathcal{C}$ is the category of representations of a finite-dimensional Hopf algebra.

In Section~\ref{sec:applications-1}, we apply our techniques to investigate further properties of the distinguished invertible object. Here we give a new proof of the Radford $S^4$-formula (\S\ref{subsec:radford-s4}), determine when $L$ and $R$ are faithful (Theorem~\ref{thm:induc-faith}), and introduce a formula expressing $D$ as an end of a certain functor (Lemma \ref{lem:end-formula}). As an application of the formula of $D$, we show that every semisimple finite tensor category ($=$ a multi-fusion category \cite{MR2183279}) is unimodular, as conjectured in \cite{2013arXiv1312.7188D}.

In Section~\ref{sec:applications-2}, we apply our results to study the role of the unimodularity in the constructions of some topological invariants (for this reason, in this section, we always assume $\End_{\mathcal{C}}(\unitobj) \cong k$). It is known that $B := R(\unitobj)$ is a commutative algebra in $\mathcal{Z}(\mathcal{C})$. The unimodularity can be characterized by this algebra:

\begin{theorem*}[Theorem~\ref{thm:com-Frob-ZC}]
  $B$ is Frobenius if and only if $\mathcal{C}$ is unimodular.
\end{theorem*}

As a first application of this theorem, in \S\ref{subsec:hb-link-inv}, we give a categorical understanding of Ishii and Masuoka's construction of handlebody-link invariants \cite{MR3265394} by generalizing their construction in the setting of unimodular finite tensor categories. The second application concerns the object $\mathrm{Int}(\mathbf{F}_{\mathcal{C}})$ of integrals of a certain Hopf algebra $\mathbf{F}_{\mathcal{C}}$ in a braided finite tensor category $\mathcal{C}$, which is used to construct 3-manifold invariants in \cite{MR1862634,MR2251160}. We prove:

\begin{theorem*}[Theorem~\ref{thm:int-F-D}]
  $\mathrm{Int}(\mathbf{F}_{\mathcal{C}}) \cong D^*$.
\end{theorem*}

Finally, we suppose that $\mathcal{C}$ is a unimodular ribbon finite tensor category. Then the above theorem implies that a non-zero integral $\Lambda$ of $\mathbf{F}_{\mathcal{C}}$ is an algebraic Kirby element (in the sense of Virelizier \cite{MR2251160}). Hence $\Lambda$ gives rise to a closed 3-manifold invariant that generalizes the Hennings-Kauffman-Radford invariant constructed from a finite-dimensional unimodular ribbon Hopf algebra (Remark~\ref{rem:HKR-inv}).

\section*{Acknowledgments}

A part of this work is done during a visit of the author to Universit\'e de Bourgogne in November 2013. The author is grateful to Peter Schauenburg for his hospitality and for helpful discussion. The author also would like to thank Atsushi Ishii, Akira Masuoka, and Taiki Shibata for helpful discussion. The author is supported by Grant-in-Aid for JSPS Fellows (24$\cdot$3606). 

\section{Preliminaries}
\label{sec:preliminaries}

\subsection{Monoidal categories}
\label{subsec:moncat}

For the basic theory of monoidal categories, we refer the reader to \cite{MR1797619,MR1321145,MR1712872}. We first fix some conventions used throughout this paper. In view of Mac Lane's coherence theorem, we may, and do, assume that all monoidal categories are strict. Given a monoidal category $\mathcal{C} = (\mathcal{C}, \otimes, \unitobj)$ with tensor product $\otimes$ and unit object $\unitobj \in \mathcal{C}$, we set
\begin{equation*}
  \mathcal{C}^{\op} = (\mathcal{C}^{\op}, \otimes, \unitobj)
  \text{\quad and \quad}
  \mathcal{C}^{\rev} = (\mathcal{C}, \otimes^{\rev}, \unitobj),
\end{equation*}
where $\mathcal{M}^{\op}$ for a category $\mathcal{M}$ means the opposite category and $\otimes^{\rev}$ is the reversed tensor product given by $V \otimes^{\rev} W = W \otimes V$ for $V, W \in \mathcal{C}$.

Let $\mathcal{C}$ and $\mathcal{D}$ be monoidal categories. A {\em monoidal functor} from $\mathcal{C}$ to $\mathcal{D}$ is a functor $F: \mathcal{C} \to \mathcal{D}$ endowed with a morphism $F_0: \unitobj \to F(\unitobj)$ and a natural transformation
\begin{equation*}
  F_2(V, W): F(V) \otimes F(W) \to F(V \otimes W)
  \quad (V, W \in \mathcal{C})
\end{equation*}
satisfying certain axioms \cite[XI.2]{MR1712872}. If $F_0$ and $F_2$ are invertible, then $F$ is said to be {\em strong}. A {\em comonoidal functor} is a monoidal functor from $\mathcal{C}^{\op}$ to $\mathcal{D}^{\op}$.

Following \cite{MR1321145}, a {\em left dual object} of $V \in \mathcal{C}$ is an object $V^* \in \mathcal{C}$ endowed with morphisms $\eval_V: V^* \otimes V \to 1$ and  $\coev_V: \unitobj \to V \otimes V^*$ in $\mathcal{C}$ such that
\begin{gather*}
  (\coev_V \otimes \id_V) (\id_V \otimes \eval_V) = \id_V
  \quad \text{and} \quad
  (\eval_V \otimes \id_{V^*}) (\id_{V^*} \otimes \coev_V) = \id_{V^*}.
\end{gather*}
One can extend $V \mapsto V^*$ to a strong monoidal functor $(-)^*: \mathcal{C}^{\op} \to \mathcal{C}^{\rev}$, called the {\em left duality functor}, provided that every object of $\mathcal{C}$ has a left dual object. A {\em right dual object} ${}^* V$ of $V \in \mathcal{C}$ is a left dual object of $V$ in $\mathcal{C}^{\rev}$. Similarly to the above, one can extend $V \mapsto {}^* V$ to a strong monoidal functor ${}^* (-): \mathcal{C}^{\op} \to \mathcal{C}^{\rev}$ if every object of $\mathcal{C}$ has a right dual object.

A monoidal category $\mathcal{C}$ is said to be {\em rigid} (or {\em autonomous}) if every object of $\mathcal{C}$ has both a left and a right dual object. If this is the case, the contravariant endofunctors $(-)^*$ and ${}^* (-)$ on $\mathcal{C}$ are mutually quasi-inverse. Moreover,  by replacing $\mathcal{C}$ with an equivalent one, we can choose dual objects so that
\begin{equation*}
  \unitobj^* = \unitobj, \quad
  (V \otimes W)^* = W^* \otimes V^*
  \text{\quad and \quad}
  {}^* (V^*) = V = ({}^* V)^*
\end{equation*}
hold for all $V, W \in \mathcal{C}$ \cite{2013arXiv1309.4539S}. Thus, throughout this paper, we always assume that these equations hold.

\subsection{Monoidal center}
\label{subsec:mon-center}

Let $\mathcal{C}$ be a monoidal category. A {\em half-braiding} for $V \in \mathcal{C}$ is a natural isomorphism $\sigma_V: V \otimes (-) \to (-) \otimes V$ such that
\begin{equation*}
  \sigma_V(X \otimes Y) = (\id_X \otimes \sigma_V(Y)) \circ (\sigma_V(X) \otimes \id_Y)
\end{equation*}
holds for all $X, Y \in \mathcal{C}$. The {\em monoidal center} (or the {\em center} for short) of $\mathcal{C}$ is the category $\mathcal{Z}(\mathcal{C})$ whose objects are the pairs $(V, \sigma_V)$, where $V \in \mathcal{C}$ and $\sigma_V$ is a half-braiding for $V$, and whose morphisms are the morphisms in $\mathcal{C}$ compatible with the half-braidings. The category $\mathcal{Z}(\mathcal{C})$ has a natural structure of a braided monoidal category; see, {\it e.g.}, \cite[XIII.4]{MR1321145}.

\subsection{Algebras in a monoidal category}
\label{subsec:algebras}

An algebra ($=$ a monoid \cite{MR1712872}) in a monoidal category $\mathcal{C}$ is an object of $\mathcal{C}$ endowed with morphisms $m_A: A \otimes A \to A$ and $u_A: \unitobj \to A$ obeying the associative law and the unit law. The morphisms $m_A$ and $u_A$ are called the {\em multiplication} and the {\em unit} of $A$, respectively.

Given an algebra $A$ in $\mathcal{C}$, we denote by ${}_A \mathcal{C}$ and $\mathcal{C}_A$ the categories of left $A$-modules and right $A$-modules, respectively. If $M$ is a left $A$-module whose underlying object is left rigid, then its left dual object $M^*$ is a right $A$-module with action
\begin{equation*}
  M^* \otimes A
  \xrightarrow{\quad \rho_M^* \quad}
  (A \otimes M)^* \otimes A
  = M^* \otimes A^* \otimes A
  \xrightarrow{\quad \id_M^* \otimes \eval_A \quad} M^*,
\end{equation*}
where $\rho: A \otimes M \to M$ is the left action of $A$ on $M$. Similarly, a right dual object of a right $A$-module has a structure of a left $A$-module.

Now let $B$ be another algebra in $\mathcal{C}$. If $X \in {}_A \mathcal{C}$ and $Y \in \mathcal{C}_B$, then their tensor product $X \otimes Y$ is an $A$-$B$-bimodule. This construction gives rise to a bifunctor
\begin{equation*}
  {}_A \mathcal{C} \times \mathcal{C}_B \to {}_A \mathcal{C}_B,
  \quad (X, Y) \mapsto X \otimes Y
  \quad (X \in {}_A \mathcal{C}, Y \in \mathcal{C}_B),
\end{equation*}
where ${}_A \mathcal{C}_B$ denotes the category of $A$-$B$-bimodules. For simplicity, we now suppose that $\mathcal{C}$ is rigid. The following lemma is well-known:

\begin{lemma}
  \label{lem:adj-restriction}
  Let $F_A: {}_A \mathcal{C}_B \to \mathcal{C}_B$ and $F_B: {}_A \mathcal{C}_B \to {}_A \mathcal{C}$ be the functors forgetting the actions of $A$ and $B$, respectively. Then:
  \begin{enumerate}
  \item ${}_A A \otimes (-)$ is left adjoint to $F_A$
  \item ${}^*(A_A) \otimes (-)$ is right adjoint to $F_A$.
  \item $(-) \otimes B_B$ is left adjoint to $F_B$
  \item $(-) \otimes ({}_B B)^*$ is right adjoint to $F_B$.
  \end{enumerate}
  Here, given an algebra $A$ in $\mathcal{C}$, we denote by ${}_A A$ and $A_A$ the object $A$ viewed as a left $A$-module and a right $A$-module by the multiplication of $A$, respectively.
\end{lemma}

For an object $K$ and an algebra $A$ in $\mathcal{C}$, there is a bijection
\begin{equation*}
  \Theta: \Hom_A(A_A, K \otimes ({}_A A)^*)
  \xrightarrow[\text{ Lemma~\ref{lem:adj-restriction} }]{\cong} \Hom_{\mathcal{C}}(\unitobj, K \otimes A^*)
  \xrightarrow{\quad \cong \quad} \Hom_{\mathcal{C}}(A, K)
\end{equation*}
A morphism $\lambda: A \to K$ is called a {\em $K$-valued trace} if $\Theta^{-1}(\lambda)$ is an isomorphism of right $A$-modules. Now we suppose that $A$ has a $K$-valued trace $\lambda$.

\begin{lemma}
  \label{lem:Nakayama-iso}
  Let $A$, $K$ and $\lambda$ be as above, and set $\phi = (\Theta^{-1}(\lambda))^{-1}$. Then
  \begin{equation*}
    \nu: A \xrightarrow{\quad \phi^{-1} \quad} K \otimes A^*
    \xrightarrow{\quad \id_K \otimes \phi^* \quad} K \otimes A^{**} \otimes K^*
  \end{equation*}
  is an isomorphism of algebras in $\mathcal{C}$.
\end{lemma}
\begin{proof}
  We denote the multiplication of $A$ and $A' := K \otimes A^{**} \otimes K^*$ by $m$ and $m'$, respectively. By definition, $m'$ is given by
  \begin{equation*}
    m' = (\id_K \otimes m^{**} \otimes \id_{K}^*) \circ (\id_K \otimes \id_A^{**} \otimes \eval_K \otimes \id_{A}^{**} \otimes \id_{K}^*).
  \end{equation*}
  Since $\phi$ is an isomorphism of right $A$-modules, we have
  \begin{equation}
    \label{eq:lem-Nak-iso-pf-1}
    (\phi \otimes \eval_A) \circ (\id_K \otimes m^* \otimes \id_A) = m \circ (\phi \otimes \id_A).
  \end{equation}
  Translating~\eqref{eq:lem-Nak-iso-pf-1} via $\Hom_{\mathcal{C}}(K \otimes A^* \otimes A, A) \cong \Hom_{\mathcal{C}}(K \otimes A^*, A \otimes A^*)$, we get
  \begin{equation}
    \label{eq:lem-Nak-iso-pf-2}
    (\phi \otimes \id_A^*) \circ (\id_K \otimes m^*) = (m \otimes \id_A^*) \circ (\phi \otimes \coev_A).
  \end{equation}
  Applying the left duality functor to~\eqref{eq:lem-Nak-iso-pf-2}, we get
  \begin{equation}
    \label{eq:lem-Nak-iso-pf-3}
    (m^{**} \otimes \id_K^*) \circ (\id_A^{**} \otimes \phi^*) = (\eval_{A^*} \otimes \phi^*) \circ (\id_A^{**} \otimes m^{**}).
  \end{equation}
  One can verify $m' \circ (\nu \otimes \nu) = \nu \circ m$ directly by using \eqref{eq:lem-Nak-iso-pf-1}--\eqref{eq:lem-Nak-iso-pf-3}. Figure~\ref{fig:Nakayama-iso} explains the details of the computation graphically (we read string diagrams from the top to the bottom and express the evaluation and the coevaluation by a cup and a cap, respectively). It is obvious that $\nu$ is invertible. Thus that $\nu$ preserves the unit follows from the uniqueness of the unit.
\end{proof}

\begin{figure}
  \begin{gather*}
    \begin{array}{c}
      \includegraphics{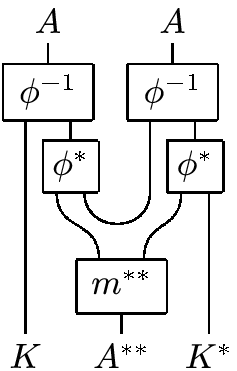}
    \end{array}
    \overset{\eqref{eq:lem-Nak-iso-pf-3}}{=}
    \begin{array}{c}
      \includegraphics{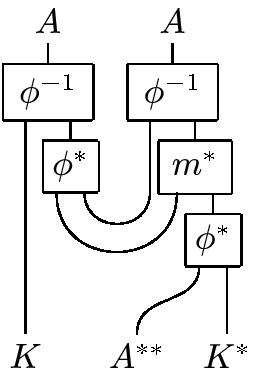}
    \end{array}
    =
    \begin{array}{c}
      \includegraphics{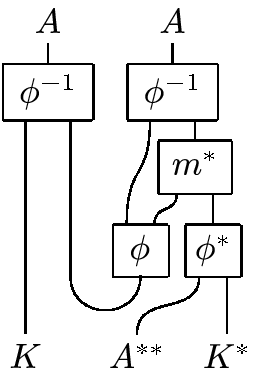}
    \end{array} \\
    \overset{\eqref{eq:lem-Nak-iso-pf-2}}{=}
    \begin{array}{c}
      \includegraphics{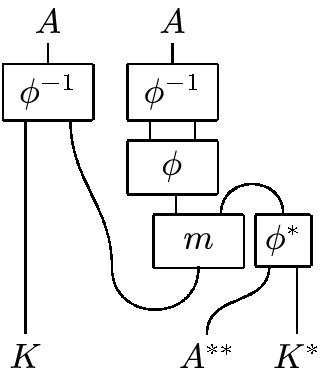}
    \end{array}
    =
    \begin{array}{c}
      \includegraphics{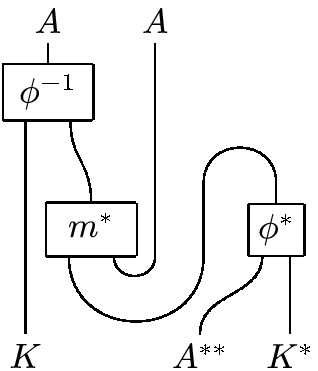}
    \end{array}
    \overset{\eqref{eq:lem-Nak-iso-pf-1}}{=}
    \begin{array}{c}
      \includegraphics{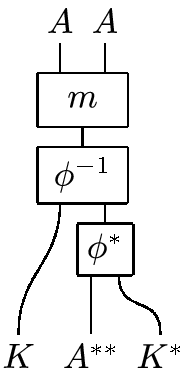}
    \end{array}
  \end{gather*}
  \caption{The proof of Lemma~\ref{lem:Nakayama-iso}}
  \label{fig:Nakayama-iso}
\end{figure}

A $\unitobj$-valued trace is simply called a {\em trace}. Recall that a functor $F$ is said to be {\em Frobenius} \cite{MR1926102} if it has a left adjoint functor which is also right adjoint to $F$. The following result is an immediate consequence of Lemma~\ref{lem:adj-restriction}.

\begin{lemma}
  \label{lem:Frobenius-2}
  For an algebra $A$ in $\mathcal{C}$, the following assertions are equivalent:
  \begin{enumerate}
  \item A trace for $A$ exists.
  \item $A {}_A \cong ({}_A A)^*$ as right $A$-modules.
  \item The forgetful functor $\mathcal{C}_A \to \mathcal{C}$ is a Frobenius functor.
  \end{enumerate}
\end{lemma}

A {\em Frobenius algebra} in $\mathcal{C}$ is a pair $(A, \lambda)$ consisting of an algebra $A$ in $\mathcal{C}$ and a trace $\lambda: A \to \unitobj$. By abuse of terminology, we often say that an algebra $A$ is Frobenius if the equivalent conditions of Lemma~\ref{lem:Frobenius-2} are satisfied.

\subsection{Colax-lax adjunctions}
\label{subsec:colax-lax-adj}

The category $\mathbf{Set}$ of all sets is a monoidal category with respect to the Cartesian product and with unit object the set $\{ * \}$ consisting of one element. Now let $\mathcal{A}$, $\mathcal{B}$ and $\mathcal{C}$ be monoidal categories. If $P: \mathcal{A} \to \mathcal{C}$ is a comonoidal functor and $Q: \mathcal{B} \to \mathcal{C}$ is a monoidal functor, then
\begin{equation*}
  H: \mathcal{A}^{\op} \times \mathcal{B} \to \mathbf{Set},
  \quad (V, W) \mapsto \Hom_{\mathcal{C}}(P(V), Q(W))
  \quad (V \in \mathcal{A}, W \in \mathcal{B})
\end{equation*}
has a structure of a monoidal functor given by $H_0(*) = Q_0 \circ P_0$ and
\begin{gather*}
  H_2: H(V, W) \times H(X, Y) 
  \to H(V \otimes X, W \otimes Y), \\
  (f, g) \mapsto Q_2(W,Y) \circ (f \otimes g) \circ P_2(V,X).
\end{gather*}
Following Mac Lane \cite[IV]{MR1712872}, we write
\begin{equation}
  \label{eq:notation-adj}
  \langle F, G, \eta, \varepsilon \rangle: \mathcal{B} \rightharpoonup \mathcal{C}
\end{equation}
if $F: \mathcal{B} \to \mathcal{C}$ is a functor, $G$ is right adjoint to $F$, and $\eta$ and $\varepsilon$ are the unit and the counit of the adjunction, respectively. We say that \eqref{eq:notation-adj} is a {\em colax-lax adjunction} \cite[\S3.9.1]{MR2724388} if $F$ is comonoidal, $G$ is monoidal and the natural isomorphism
\begin{equation*}
  \Hom_{\mathcal{C}}(F(V), W) \cong \Hom_{\mathcal{C}}(V, G(W))
  \quad (V \in \mathcal{B}, W \in \mathcal{C})
\end{equation*}
of the adjunction is an isomorphism of monoidal functors from $\mathcal{B}^{\op} \times \mathcal{C}$ to $\mathbf{Set}$. This notion is in fact an instance of doctrinal adjunctions \cite{MR0360749} and therefore we have the following result (see \cite[\S3.9.1]{MR2724388} for details):

\begin{lemma}
  \label{lem:colax-lax-adj-1}
  Let $\langle F, G, \eta, \varepsilon \rangle: \mathcal{B} \rightharpoonup \mathcal{C}$ be an adjunction between monoidal categories $\mathcal{B}$ and $\mathcal{C}$. If $F$ is comonoidal (respectively, $G$ is monoidal), then there uniquely exists a monoidal structure of $G$ (respectively, a comonoidal structure of $F$) such that $\langle F, G, \eta, \varepsilon \rangle$ is a colax-lax adjunction.
\end{lemma}

An adjoint functor is often determined only up to isomorphism. Thus we consider the case where a functor $F: \mathcal{B} \to \mathcal{C}$ has two right adjoint functors $G$ and $G'$. Then there is a canonical isomorphism $G \cong G'$ induced from
\begin{equation}
  \label{eq:colax-lax-two-adj}
  \Hom_{\mathcal{C}}(V, G(W)) \cong \Hom_{\mathcal{B}}(F(V), W) \cong \Hom_{\mathcal{C}}(V, G'(W)).
\end{equation}
If $F$ is comonoidal, then both $G$ and $G'$ are monoidal by Lemma~\ref{lem:colax-lax-adj-1}. Since the isomorphisms in \eqref{eq:colax-lax-two-adj} are monoidal, the canonical isomorphism $G \cong G'$ is in fact an isomorphism of monoidal functors. Similarly, two left adjoint functors of a monoidal functor are canonically isomorphic as comonoidal functors.

Now suppose that $\mathcal{B}$ and $\mathcal{C}$ are rigid. For a functor $T: \mathcal{B} \to \mathcal{C}$, we define $T^!$ to be the following composition of functors:
\begin{equation*}
  T^!: \mathcal{B} \xrightarrow{\quad (-)^* \quad} \mathcal{B}^{\op}
  \xrightarrow{\quad T^{\op} \quad} \mathcal{C}^{\op}
  \xrightarrow{\quad {}^*(-) \quad} \mathcal{C}.
\end{equation*}
If $F: \mathcal{B} \to \mathcal{C}$ is strong monoidal, then there is an isomorphism $F^! \cong F$ of monoidal functors \cite[Lemma 1.1]{MR2381536}. If, moreover, $L$ is left adjoint to $F$, then $L^!$ is right adjoint to $F$ \cite[Lemma 3.5]{MR2869176}. Indeed, we have isomorphisms
\begin{equation}
  \label{eq:colax-lax-L!}
  \begin{aligned}
    \Hom_{\mathcal{C}}(V, L^!(W))
    & \cong \Hom_{\mathcal{C}}(L(W^*), V^*) \\
    & \cong \Hom_{\mathcal{C}}(W^*, F(V^*)) \\
    & \cong \Hom_{\mathcal{C}}(F^!(V), W)  \cong \Hom_{\mathcal{C}}(F(V), W)
  \end{aligned}
\end{equation}
natural in the variables $V \in \mathcal{C}$ and $W \in \mathcal{B}$. Similarly, if $R$ is right adjoint to $F$, then $R^!$ is left adjoint to $F$.

By Lemma~\ref{lem:colax-lax-adj-1}, $L$ is a comonoidal functor. Hence $L^!$ is a monoidal functor with monoidal structure given by ${}^*L_0: \unitobj \to L^!(\unitobj)$ and
\begin{equation*}
  L^!(X) \otimes L^!(Y) = {}^*(L(Y^*) \otimes L(X^*))
  \xrightarrow{\quad {}^*L_2(X,Y) \quad}
  {}^*L(Y^* \otimes X^*) = L^!(X \otimes Y),
\end{equation*}
where $L_0$ and $L_2$ are the comonoidal structure of $L$. On the other hand, since $L^!$ is right adjoint to $F$, it has another monoidal structure by Lemma~\ref{lem:colax-lax-adj-1}. The following lemma says that these two structures are the same.

\begin{lemma}
  \label{lem:colax-lax-adj-2}
  Let $F: \mathcal{B} \to \mathcal{C}$ a strong monoidal functor between rigid monoidal categories. Suppose that $F$ has a left adjoint $L$ and a right adjoint $R$. Then the canonical isomorphism $L^! \cong R$ is an isomorphism of monoidal functors.
\end{lemma}
\begin{proof}
  The isomorphism $\Hom_{\mathcal{B}}(F(V), W) \cong \Hom_{\mathcal{C}}(V, L^!(W))$ obtained in the above is in fact an isomorphism of monoidal functors. Hence
  \begin{equation*}
    \Hom_{\mathcal{C}}(V, R(W)) \cong \Hom_{\mathcal{B}}(F(V), W) \cong \Hom_{\mathcal{C}}(V, L^!(W))
  \end{equation*}
  as monoidal functors from $\mathcal{C}^{\op} \times \mathcal{B}$ to $\mathbf{Set}$. Now the result follows from the Yoneda lemma.
\end{proof}

Applying Lemma~\ref{lem:colax-lax-adj-2} to the functor $F^{\rev}: \mathcal{B}^{\rev} \to \mathcal{C}^{\rev}$ induced by $F$, we also have an isomorphism $R \cong {}^! L$ of monoidal functors, where ${}^!L = L(^*-)^*$.

Since $R$ is monoidal, $A = R(\unitobj)$ is an algebra in $\mathcal{C}$ as the image of the trivial algebra $\unitobj \in \mathcal{C}$. Similarly, since $L$ is comonoidal, $C = L(\unitobj)$ is a coalgebra in $\mathcal{C}$. The above lemma implies that $A \cong {}^*C$ as algebras in $\mathcal{C}$.

\subsection{Hopf monads}
\label{subsec:Hopf-monads}

Let $T = (T, \mu, \eta)$ be a monad \cite[VI.1]{MR1712872} on a category $\mathcal{C}$ with multiplication $\mu$ and unit $\eta$. By a {\em $T$-module}, we mean an object $M \in \mathcal{C}$ endowed with a morphism $\rho_M: T(M) \to M$ satisfying
\begin{equation*}
  \rho_M \circ \mu_M = \rho_M \circ T(\rho_M)
  \text{\quad and \quad}
  \rho_M \circ \eta_M = \id_M.
\end{equation*}
This notion is also called a ``$T$-algebra'' in literature but we do not use this term in this paper. We denote by ${}_T \mathcal{C}$ the category of $T$-modules ($=$ the Eilenberg-Moore category of $T$-algebras \cite[VI.2]{MR1712872}).

Let $\mathcal{C}$ be a monoidal category. A {\em bimonad} \cite{MR2355605,MR2793022} on $\mathcal{C}$ is a monad $T = (T, \mu, \eta)$ on $\mathcal{C}$ such that the functor $T$ is comonoidal and the natural transformations $\mu$ and $\eta$ are comonoidal natural transformations. If $M$ and $N$ are $T$-modules, then their tensor product $M \otimes N$ is also a $T$-module by
\begin{equation*}
  T(M \otimes N) \xrightarrow{\quad T_2 \quad} T(M) \otimes T(N) \xrightarrow{\quad \rho_M \otimes \rho_N \quad} M \otimes N,
\end{equation*}
where $\rho_M$ and $\rho_N$ are the action of $T$ on $M$ and $N$, respectively. The category ${}_T \mathcal{C}$ of $T$-modules is a monoidal category with this operation.

Now suppose that $\mathcal{C}$ is rigid. Then a {\em Hopf monad} on $\mathcal{C}$ is a bimonad $T$ on $\mathcal{C}$ endowed with natural transformations
\begin{equation*}
  S_V: T(T(V)^*) \to V^*
  \quad \text{and} \quad
  \overline{S}_V: T({}^*T(V)) \to {}^*V \quad (V \in \mathcal{C})
\end{equation*}
satisfying certain conditions. The natural transformations $S$ and $\overline{S}$ are called the left antipode and the right antipode of $T$, respectively. If $T$ is a Hopf monad on $\mathcal{C}$, then the monoidal category ${}_T \mathcal{C}$ is rigid. The left dual object of $M \in {}_T \mathcal{C}$ is the left dual object $M^*$ in $\mathcal{C}$ with the action given by
\begin{equation}
  \label{eq:Hopf-monad-dual-module}
  T(M^*) \xrightarrow{\ T(\rho_M^*) \ } T(T(M)^*) \xrightarrow{\ S_M \ } M^*.
\end{equation}

\subsection{Finite abelian categories}

Let $k$ be a field. Given a $k$-algebra $A$, we denote by $A\mbox{-{\sf mod}}$ and $\mbox{{\sf mod}-}A$ the categories of finite-dimensional left and right $A$-modules, respectively. The following variant of the Eilenberg-Watts theorem \cite{MR0125148,MR0118757} will be used extensively:

\begin{lemma}
  \label{lem:EW-thm}
  Let $A$ and $B$ be finite-dimensional $k$-algebras. For a $k$-linear functor $F: \mbox{{\sf mod}-}A \to \mbox{{\sf mod}-}B$, the following three assertions are equivalent:
  \begin{enumerate}
    \renewcommand{\labelenumi}{(\arabic{enumi})}
  \item $F$ is left exact.
  \item $F$ has a left adjoint.
  \item $F \cong \Hom_A(M, -)$ for some finite-dimensional $B$-$A$-bimodule $M$.
  \end{enumerate}
  The following three assertions are also equivalent:
  \begin{enumerate}
    \renewcommand{\labelenumi}{(\arabic{enumi})$'$}
  \item $F$ is right exact.
  \item $F$ has a right adjoint.
  \item $F \cong (-) \otimes_A M$ for some finite-dimensional $A$-$B$-bimodule $M$.
  \end{enumerate}
\end{lemma}

By a {\em finite abelian category} over $k$, we mean a $k$-linear abelian category equivalent to $\mbox{{\sf mod}-}A$ for some finite-dimensional $k$-algebra $A$.

\begin{lemma}
  \label{lem:finite-T-mod}
  Let $\mathcal{A}$ be a finite abelian category over $k$, and let $T$ be a $k$-linear right exact monad on $\mathcal{A}$. Then ${}_T \mathcal{A}$ is also a finite abelian category over $k$.
\end{lemma}
\begin{proof}
  Following Eilenberg and Moore \cite[Proposition 5.3]{MR0184984}, the category ${}_T \mathcal{A}$ is a $k$-linear abelian category such that the forgetful functor $F: {}_T \mathcal{A} \to \mathcal{A}$ preserves and reflects exact sequences. To complete the proof, it is enough to show that ${}_T \mathcal{A}$ has a projective generator. Let $L$ be a left adjoint of $F$, and let $P \in \mathcal{A}$ be a projective generator. Then $Q = L(P) \in {}_T \mathcal{A}$ is projective, since
  \begin{equation*}
    \Hom_{T}(L(P), -) \cong \Hom_{\mathcal{A}}(P, F(-)) = \Hom_{\mathcal{A}}(P, -) \circ F
  \end{equation*}
  is exact. Now let $M \in {}_T \mathcal{A}$. Then there exists an epimorphism $f: P^{\oplus m} \to F(M)$ in $\mathcal{A}$. Note that $L$ preserves epimorphisms as it is left adjoint. Since $U$ is faithful, the counit $\varepsilon$ of the adjunction is epic \cite[IV.3]{MR1712872}. Hence the composition
  \begin{equation*}
    \begin{CD}
      Q^{\oplus m} = L(P^{\oplus m}) @>{L(f)}>> L U(M) @>{\varepsilon}>> M
    \end{CD}
  \end{equation*}
  is epic. Therefore $Q \in {}_T \mathcal{A}$ is a projective generator.
\end{proof}

\subsection{Finite tensor categories}
\label{subsec:FTC}

Following \cite{MR2119143}, a {\em finite tensor category} over $k$ is a rigid monoidal category $\mathcal{C}$ such that $\mathcal{C}$ is a non-zero finite abelian category over $k$ and the tensor product $\otimes: \mathcal{C} \times \mathcal{C} \to \mathcal{C}$ is $k$-linear in each variable.

We note that the tensor product of a finite tensor category $\mathcal{C}$ is exact in each variable \cite[Proposition 2.18]{MR1797619}, since there are adjunctions
\begin{equation*}
  V^* \otimes (-) \ \dashv \  V \otimes (-) \ \dashv \  {}^* V \otimes (-)
  \text{\quad and \quad}
  (-) \otimes {}^*V \ \dashv \  (-) \otimes V \ \dashv \  (-) \otimes V^*
\end{equation*}
for each $V \in \mathcal{C}$, where $F \dashv G$ means that $F$ is a left adjoint functor of $G$.

Unlike \cite{MR2119143}, and like \cite{2014arXiv1406.4204D}, we do not assume that the unit object $\unitobj \in \mathcal{C}$ is a simple object (thus our finite tensor category is in fact a {\em finite multi-tensor category} in the sense of \cite{MR2119143}). It is known that $\unitobj \in \mathcal{C}$ can be written as the direct sum
\begin{equation}
  \label{eq:unit-decomp-1}
  \unitobj = \unitobj_1 \oplus \dotsb \oplus \unitobj_m
\end{equation}
of pairwise non-isomorphic simple objects $\unitobj_1, \dotsc, \unitobj_m \in \mathcal{C}$ such that
\begin{equation}
  \label{eq:unit-decomp-2}
  \unitobj_i \otimes \unitobj_i \cong \unitobj_i,
  \quad \unitobj_i \otimes \unitobj_j = 0 \quad (i \ne j),
  \quad \text{and} \quad \unitobj_i^* \cong \unitobj_i
\end{equation}
for all $i, j = 1, \dotsc, m$. In particular, $\unitobj$ is a semisimple object. Thus:

\begin{lemma}
  The unit object $\unitobj \in \mathcal{C}$ is a simple object if $\End_{\mathcal{C}}(\unitobj) \cong k$.
\end{lemma}

The full subcategories $\mathcal{C}_{i j} := \unitobj_i \otimes \mathcal{C} \otimes \unitobj_j \subset \mathcal{C}$ are called the {\em component  subcategories} of $\mathcal{C}$. Now suppose that $X \in \mathcal{C}_{p q}$ and $Y \in \mathcal{C}_{r s}$. If $q \ne r$, then $X \otimes Y = 0$ by~\eqref{eq:unit-decomp-2}. Otherwise, we have an inequality
\begin{equation}
  \label{eq:len-ineq}
  \ell(X \otimes Y) \ge \ell(X) \cdot \ell(Y),
\end{equation}
where $\ell(M)$ denotes the length of $M$. In particular, $X \otimes Y \ne 0$ if $X \in \mathcal{C}_{p q}$ and $Y \in \mathcal{C}_{q r}$ are non-zero objects; see \cite{EGNO-Lect} for more details.

\subsection{Finite module categories}
\label{subsec:FMC}

Let $\mathcal{C}$ be a monoidal category. A {\em left $\mathcal{C}$-module category} is a category $\mathcal{M}$ endowed with a functor $\ogreaterthan: \mathcal{C} \times \mathcal{M} \to \mathcal{M}$, called the {\em action} of $\mathcal{C}$, and natural isomorphisms
\begin{equation*}
  \unitobj \ogreaterthan M \cong M \text{\quad and \quad}
  (X \otimes Y) \ogreaterthan M \cong X \ogreaterthan (Y \ogreaterthan M)
  \quad (X, Y \in \mathcal{C}, M \in \mathcal{M})
\end{equation*}
satisfying the axioms similar to those for monoidal categories. See \cite{MR1976459} for the precise definitions of a left $\mathcal{C}$-module category and related notions.

Now suppose that $\mathcal{C}$ is a finite tensor category over a field $k$. We say that a left $\mathcal{C}$-module category $\mathcal{M}$ is {\em finite} if its underlying category is a finite abelian category over $k$ and the action $\ogreaterthan: \mathcal{C} \times \mathcal{M} \to \mathcal{M}$ of $\mathcal{C}$ is $k$-linear in each variable and right exact in the first variable. Note that the action $\ogreaterthan$ is always exact in the second variable since, for each $V \in \mathcal{C}$, there are adjunctions
\begin{equation*}
  V^* \ogreaterthan (-)
  \quad \dashv \quad V \ogreaterthan (-)
  \quad \dashv \quad {}^* V \ogreaterthan (-).
\end{equation*}

Let $\mathcal{M}$ be a finite left $\mathcal{C}$-module category. For $M \in \mathcal{M}$, the functor $(-) \ogreaterthan M$ has a right adjoint by Lemma~\ref{lem:EW-thm}. We denote it by $\iHom(M, -)$. By definition, there is an isomorphism of vector spaces
\begin{equation}
  \label{eq:internal-Hom-def}
  \Hom_{\mathcal{C}}(V, \iHom(M, N)) \cong \Hom_{\mathcal{M}}(V \ogreaterthan M, N)
\end{equation}
natural in the variables $V$ and $N$. The assignment $(M, N) \mapsto \iHom(M, N)$ uniquely extends to a functor $\iHom: \mathcal{M}^{\op} \times \mathcal{M} \to \mathcal{C}$, called the {\em internal Hom functor} for $\mathcal{M}$, in such a way that \eqref{eq:internal-Hom-def} is natural also in the variables $M$ and $N$.

The counit of the adjunction $(-) \ogreaterthan M \dashv \iHom(M, -)$, denoted by
\begin{equation}
  \label{eq:internal-Hom-eval}
  \ieval_{M,N}: \iHom(M, N) \ogreaterthan M \to N
  \quad (N \in \mathcal{M}),
\end{equation}
is often called the {\em evaluation}. For $L, M, N \in \mathcal{M}$, the {\em composition}
\begin{equation}
  \label{eq:internal-Hom-composition}
  \icomp_{L,M,N}: \iHom(M, N) \otimes \iHom(L, M) \to \iHom(L, N)
\end{equation}
is defined to be the morphism corresponding to
\begin{align*}
  (\iHom(M, N) & \otimes \iHom(L, M)) \ogreaterthan L \\[-8pt]
  \cong \iHom(M, N) & \ogreaterthan (\iHom(L, M) \ogreaterthan L)
  \xrightarrow{\ \id \ogreaterthan \ieval_{L,M} \ } \iHom(M, N) \ogreaterthan M
  \xrightarrow{\ \ieval_{M,N} \ } N  
\end{align*}
via natural isomorphism \eqref{eq:internal-Hom-def}, and the {\em identity}
\begin{equation}
  \label{eq:internal-Hom-unit}
  \iId_M: \unitobj \to \iEnd(M) \quad (=\iHom(M, M))
\end{equation}
is the morphism corresponding to the isomorphism $\unitobj \ogreaterthan M \cong M$ via \eqref{eq:internal-Hom-def}. The composition and the identity behave like those in a usual category ({\it i.e.}, $\mathcal{M}$ is a $\mathcal{C}$-enriched category). In particular, $\iEnd(M)$ is an algebra in $\mathcal{C}$.

\begin{example}
  \label{ex:mod-cat-1}
  Set $\mathcal{V} = \mbox{{\sf mod}-}k$. Every finite abelian category $\mathcal{M}$ over $k$ has a natural structure of a finite left $\mathcal{V}$-module category with action ``$\cdot$'' determined by
  \begin{equation*}
    \Hom_{\mathcal{A}}(V \cdot M, N) \cong \Hom_{k}(V, \Hom_{\mathcal{M}}(M, N))
    \quad (V \in \mathcal{V}, M, N \in \mathcal{M}).
  \end{equation*}
  By definition, $\iHom(M, N) = \Hom_{\mathcal{M}}(M, N)$ for all $M, N \in \mathcal{M}$. In this example, \eqref{eq:internal-Hom-eval}, \eqref{eq:internal-Hom-composition} and \eqref{eq:internal-Hom-unit} coincide with the evaluation, the composition of maps and the identity map, respectively.
\end{example}

\begin{example}
  \label{ex:mod-cat-2}
  Let $\mathcal{B}$ and $\mathcal{C}$ be a finite tensor categories, and let $F: \mathcal{B} \to \mathcal{C}$ be a $k$-linear right exact strong monoidal functor. Then $\mathcal{C}$ is a finite left $\mathcal{B}$-module category with action given by $X \ogreaterthan V = F(X) \otimes V$ ($X \in \mathcal{B}$, $V \in \mathcal{C}$). By Lemma~\ref{lem:EW-thm}, $F$ has a right adjoint functor $R$. Since
  \begin{equation*}
    \Hom_{\mathcal{C}}(X \ogreaterthan V, W) \cong \Hom_{\mathcal{C}}(F(X), W \otimes V^*) \cong \Hom_{\mathcal{B}}(X, R(W \otimes V^*),
  \end{equation*}
  the internal Hom functor is given by $\iHom(V, W) = R(W \otimes V^*)$.
\end{example}

\begin{example}
  \label{ex:mod-cat-3}
  Let $A$ be an algebra in a finite tensor category $\mathcal{C}$. The category $\mathcal{C}_A$ of right $A$-modules in $\mathcal{C}$ has a natural structure of a finite left $\mathcal{C}$-module category with action given by $X \ogreaterthan M = X \otimes M$ for $X \in \mathcal{C}$ and $M \in \mathcal{C}_A$. We have
  \begin{equation*}
    \iHom(M, N) = (M \otimes_A {}^* N)^*
    \quad (M, N \in \mathcal{C}_A),
  \end{equation*}
  where $\otimes_A$ is the tensor product over $A$ \cite[Example 2.10.8]{MR1976459}.
\end{example}

We go back to the general setting. Let $\mathcal{M}$ be a finite module category over a finite tensor category $\mathcal{C}$. Following \cite{MR1976459}, there is a natural isomorphism
\begin{equation}
  \label{eq:internal-Hom-tensor}
  \iHom(X \ogreaterthan M, Y \ogreaterthan N) \cong Y \otimes \iHom(M, N) \otimes X^*
  \quad (M, N \in \mathcal{M}, X, Y \in \mathcal{C}).
\end{equation}
For $M \in \mathcal{M}$, the monad $T$ associated to~\eqref{eq:internal-Hom-def} is given by
\begin{equation*}
  T = \iHom(M, (-) \ogreaterthan M) \cong (-) \otimes \iHom(M, M) = (-) \otimes A,
\end{equation*}
where $A = \iEnd(M)$. Thus a $T$-module is precisely a right $A$-module in $\mathcal{C}$, and the comparison functor \cite[VI.3]{MR1712872} for \eqref{eq:internal-Hom-def} is given by
\begin{equation*}
  K_M: \mathcal{M} \to \mathcal{C}_A,
  \quad N \mapsto \iHom(M, N)
  \quad (N \in \mathcal{M}),
\end{equation*}
where the action of $A$ on $\iHom(M, N)$ is given by \eqref{eq:internal-Hom-composition} with $L = M$. By~\eqref{eq:internal-Hom-tensor}, the functor $K_M$ is in fact a functor of left $\mathcal{C}$-module categories. Applying the Barr-Beck theorem \cite[VI.7]{MR1712872} to $K_M$, we obtain the following theorem:

\begin{theorem}[{\cite{EGNO-Lect}}]
  \label{thm:mod-cat-comparison}
  The functor $K_M$ above is an equivalence of left $\mathcal{C}$-module categories if and only if the following two conditions are satisfied:
  \begin{enumerate}
  \item [(K1)] The functor $\iHom(M, -): \mathcal{M} \to \mathcal{C}$ is exact.
  \item [(K2)] Every object of $\mathcal{M}$ is a quotient of an object of the form $X \ogreaterthan M$, $X \in \mathcal{C}$.
  \end{enumerate}
\end{theorem}

See also \cite{2013arXiv1312.7188D}, where a detailed proof of this theorem is given. They also studied conditions equivalent to (K1) and (K2). For later use, we note from \cite{2013arXiv1312.7188D} that the condition (K2) is equivalent to that $\iHom(M, -)$ is faithful.

\section{The central Hopf monad}
\label{thm:central-Hopf-monad}

\subsection{Ends and coends}
\label{subsec:ends}

Let $\mathcal{A}$ and $\mathcal{B}$ be categories, and let $S$ and $T$ be functors from $\mathcal{A}^{\op} \times \mathcal{A}$ to $\mathcal{B}$. A {\em dinatural transformation} $\xi: S \dinatto T$ is a family
\begin{equation*}
  \xi = \{ \xi_X: S(X, X) \to T(X, X) \}_{X \in \mathcal{A}}
\end{equation*}
of morphisms in $\mathcal{B}$ parametrized by the objects of $\mathcal{A}$ such that
\begin{equation*}
  T(\id_X, f) \circ \xi_X \circ S(f, \id_X)
  = T(f, \id_Y) \circ \xi_Y \circ S(\id_Y, f)
\end{equation*}
for all morphism $f: X \to Y$ in $\mathcal{A}$. We regard an object $X \in \mathcal{B}$ as a functor from $\mathcal{A}^{\op} \times \mathcal{A}$ to $\mathcal{B}$ sending all morphisms to $\id_X$. An {\em end} of $S$ is a pair $(E, p)$ consisting of an object $E \in \mathcal{B}$ and a dinatural transformation $p: E \dinatto S$ satisfying a certain universal property (by abuse of terminology, we also refer to $E$ and $p$ as an end). A {\em coend} of $T$ is a pair $(C, i)$ consisting of an object $C \in \mathcal{B}$ and a `universal' dinatural transformation $i: T \dinatto C$. The universal property ensures that a (co)end is unique up to isomorphism if it exists. An end $(E, p)$ of $S$ and a coend $(C, i)$ of $T$ will be denoted, respectively, by
\begin{equation*}
  E = \int_{X \in \mathcal{A}} S(X, X)
  \quad \text{and} \quad
  C = \int^{X \in \mathcal{A}} T(X, X).
\end{equation*}

We refer the reader to \cite{MR1712872} for general treatments of (co)ends. For reader's convenience, we here collect some formulas for (co)ends. Suppose that $\mathcal{A}$ is essentially small. Given two functors $F_1, F_2: \mathcal{A} \to \mathcal{B}$, we denote by $\NAT(F_1, F_2)$ the set of natural transformations from $F_1$ to $F_2$. Then
\begin{equation*}
  p_X: \NAT(F_1, F_2) \dinatto \Hom_{\mathcal{B}}(F_1(X), F_2(X)),
  \quad \alpha \mapsto \alpha_X
  \quad (X \in \mathcal{A})
\end{equation*}
is an end of $\Hom_{\mathcal{B}}(F_1(-), F_2(-))$. With integral notation, we have
\begin{equation}
  \label{eq:end-nat}
  \NAT(F_1, F_2) = \int_{X \in \mathcal{A}} \Hom_{\mathcal{B}}(F_1(X), F_2(X)).
\end{equation}

The following formula will be used extensively: If a coend of $T$ exists, then
\begin{equation}
  \label{eq:coend-hom}
  \Hom_{\mathcal{B}}(\int^{X \in \mathcal{A}} T(X,X), V)
  = \int_{X \in \mathcal{A}} \Hom_{\mathcal{B}}(T(X,X), V)
  \quad (V \in \mathcal{B}).
\end{equation}

Since the category $\mathbf{Set}$ is complete, the end of the right-hand side of \eqref{eq:coend-hom} exists without the assumption that a coend of $T$ exists. Thus, by the parameter theorem for ends \cite[IX.7]{MR1712872}, we obtain a functor
\begin{equation*}
  T^{\natural}: \mathcal{B} \to \mathbf{Set},
  \quad V \mapsto \int_{X \in \mathcal{A}} \Hom_{\mathcal{B}}(T(X, X), V)
  \quad (V \in \mathcal{B}).
\end{equation*}

\begin{lemma}
  \label{lem:coend-existence}
  A coend of $T$ exists if and only if $T^{\natural}$ is representable.
  \end{lemma}
\begin{proof}
  In view of \eqref{eq:coend-hom}, it is sufficient to show that a coend of $T$ exists if $T^{\natural}$ is representable. Let $C$ be an object representing $T^{\natural}$. By the definition of the functor $T^{\sharp}$, there exists a family
  \begin{equation*}
    \{ \phi_{X,V}: \Hom_{\mathcal{B}}(C, V) \to \Hom_{\mathcal{B}}(T(X,X), V) \}_{X \in \mathcal{A}, V \in \mathcal{B}}
  \end{equation*}
  of maps that is natural in $V$ and dinatural in $X$. By the Yoneda lemma, $\phi_{X,V}$ is induced by a morphism $i_X: T(X,X) \to C$. The family $i = \{ i_X \}$ is dinatural in $X$, and the pair $(C, i)$ is indeed a coend of $T$.
\end{proof}

\subsection{The central Hopf monad}
\label{subsec:central-Hopf-monad}

Suppose that $\mathcal{C}$ is a rigid monoidal category such that the coend
\begin{equation}
  \label{eq:Hopf-monad-Z}
  Z(V) = \int^{X \in \mathcal{C}} X^* \otimes V \otimes X
\end{equation}
exists for all $V \in \mathcal{C}$. Then the assignment $V \mapsto Z(V)$ extends to an endofunctor $Z$ on $\mathcal{C}$. Day and Street \cite{MR2342829} showed that the functor $Z$ has a structure of a monad such that ${}_Z \mathcal{C} \cong \mathcal{Z}(\mathcal{C})$ as categories. Following Brugui\`eres and Virelizier \cite{MR2869176}, the monad $Z$ has a structure of a quasitriangular Hopf monad and the isomorphism ${}_Z \mathcal{C} \cong \mathcal{Z}(\mathcal{C})$ is in fact an isomorphism of braided monoidal categories. We call the Hopf monad $Z$ the {\em central Hopf monad} on $\mathcal{C}$.

For later use, we recall from \cite{MR2342829} and \cite{MR2869176} the definition of the central Hopf monad and the construction of the isomorphism ${}_Z \mathcal{C} \cong \mathcal{Z}(\mathcal{C})$. For $V, X \in \mathcal{C}$, we denote by $i_V(X): X^* \otimes V \otimes X \to Z(V)$ the component of the universal dinatural transformation. Then the comonoidal structure
\begin{equation*}
  Z_0: Z(\unitobj) \to \unitobj
  \quad \text{and} \quad
  Z_2(V, W): Z(V \otimes W) \to Z(V) \otimes Z(W)
  \quad (V, W \in \mathcal{C})
\end{equation*}
are defined to be the unique morphisms such that $Z_0 \circ i_\unitobj(X) = \eval_X$ and
\begin{equation*}
  Z_2(V, W) \circ i_{V \otimes W}(X) = (i_V(X) \otimes i_W(X)) \circ (\id_{X^*} \otimes \id_V \otimes \coev_{X} \otimes \id_W \otimes \id_X)
\end{equation*}
for all $X \in \mathcal{C}$, respectively. The unit of $Z$ is given by $\eta_V = i_V(\unitobj)$ ($V \in \mathcal{C}$). To define the multiplication of $Z$, we note that
\begin{equation*}
  i_V^{(2)}(X, Y) := i_{Z(V)}(Y) \circ (\id_{Y^*} \otimes i_V(X) \otimes \id_Y) \quad (X, Y \in \mathcal{C})
\end{equation*}
is a coend of $(X_1, Y_1, X_2, Y_2) \mapsto X_2^* \otimes Y_2^* \otimes V \otimes X_1 \otimes Y_1$ ($X_1, X_2, Y_1, Y_2 \in \mathcal{C}$) by the Fubini theorem for coends \cite[IX.8]{MR1712872}. Hence we can define $\mu: Z^2 \to Z$ by
\begin{equation}
  \label{eq:Hopf-monad-Z-mult}
  \mu_V \circ i_V^{(2)}(X, Y) = i_V(X \otimes Y)
  \quad (V, X, Y \in \mathcal{C}).
\end{equation}
The left antipode of $Z$ is given in Remark~\ref{rem:Hopf-monad-Z-antipode} below. We omit the descriptions of the right antipode and the universal $R$-matrix of $Z$ since we will not use them.

Following \cite{MR2342829}, we establish an isomorphism ${}_Z \mathcal{C} \cong \mathcal{Z}(\mathcal{C})$ of categories. We first note that, by~\eqref{eq:end-nat} and \eqref{eq:coend-hom}, there are natural isomorphisms
\begin{align*}
  \Hom_{\mathcal{C}}(Z(V), W)
    & \cong \textstyle \int_{X \in \mathcal{C}} \Hom_{\mathcal{C}}(X^* \otimes V \otimes X, W) \\
    & \cong \textstyle \int_{X \in \mathcal{C}} \Hom_{\mathcal{C}}(V \otimes X, X \otimes W) \\
    & \cong \strut \NAT(V \otimes (-), (-) \otimes W)
\end{align*}
for $V, W \in \mathcal{C}$. Let $\partial_V(-): V \otimes (-) \to (-) \otimes Z(V)$ denote the natural transformation corresponding to $\id_{Z(V)}$ via the above chain of isomorphisms. If $V$ is a $Z$-module with action $\rho$, then one can check that
\begin{equation*}
  \sigma_V(X): V \otimes X \xrightarrow{\quad \partial_V(X) \quad} X \otimes Z(V) \xrightarrow{\quad \id_X \otimes \rho \quad} X \otimes V
  \quad (X \in \mathcal{C})
\end{equation*}
is a half-braiding for $V$. This construction gives rise to an isomorphism ${}_Z \mathcal{C} \cong \mathcal{Z}(\mathcal{C})$ of categories. The isomorphism is in fact monoidal and commutes with the forgetful functors to $\mathcal{C}$.

\begin{remark}
  \label{rem:Hopf-monad-Z-antipode}
  The Hopf monad $Z$ can be defined by using the natural transformation $\partial$ instead of the dinatural transformation $i$ (in fact, this is the way of \cite{MR2869176}). By using $\partial$, the left antipode is defined by
  \begin{equation}
    \label{eq:Hopf-monad-Z-antipode-2}
    (\id_X \otimes S_V) \circ \partial_{Z(V)^*}(X) = \partial_{V}({}^*X)^*
    \quad (V, X \in \mathcal{C}).
  \end{equation}
\end{remark}

\subsection{Existence of coends}

To apply the above Hopf monadic description of the center to finite tensor categories, we show that a coend of certain type of functors, including \eqref{eq:Hopf-monad-Z}, exists in a finite tensor category over a field $k$.

Given $k$-linear abelian categories $\mathcal{A}_1, \dotsc, \mathcal{A}_n$ and $\mathcal{C}$, we denote by
\begin{equation*}
  \LEX_n(\mathcal{A}_1, \dotsc, \mathcal{A}_n; \mathcal{C})
  \quad (\text{respectively, } \REX_n(\mathcal{A}_1, \dotsc, \mathcal{A}_n; \mathcal{C}))
\end{equation*}
the category of functors from ${\mathcal{A}_1 \times \dotsb \times \mathcal{A}_n}$ to $\mathcal{C}$ being $k$-linear left exact (respectively, right exact) in each variable. For simplicity, we write
\begin{equation*}
  \LEX(\mathcal{A}, \mathcal{C}) = \LEX_1(\mathcal{A}; \mathcal{C})
  \text{\quad and \quad}
  \REX(\mathcal{A}, \mathcal{C}) = \REX_1(\mathcal{A}; \mathcal{C}).
\end{equation*}
A {\em tensor product} \cite[\S5]{MR1106898} of $k$-linear abelian categories $\mathcal{A}_1, \dotsc, \mathcal{A}_n$ is a $k$-linear abelian category $\mathcal{T}$ endowed with $\boxtimes \in \REX_n(\mathcal{A}_1, \dotsc, \mathcal{A}_n; \mathcal{T})$ such that
\begin{equation*}
  \REX(\mathcal{T}, \mathcal{C}) \to \REX_n(\mathcal{A}_1, \dotsc, \mathcal{A}_n; \mathcal{C})
  \quad F \mapsto F \circ \boxtimes
  \quad (F \in \REX(\mathcal{T}, \mathcal{C}))
\end{equation*}
is an equivalence for any $k$-linear abelian category $\mathcal{C}$. A tensor product of $\mathcal{A}_1, \dotsc, \mathcal{A}_n$ does not always exist; see \cite{LopezFranco2013207}. If it exists, it is unique up to equivalence and is denoted by $\mathcal{A}_1 \boxtimes \dotsb \boxtimes \mathcal{A}_n$.

If $\mathcal{A} = \mbox{{\sf mod}-}A$ and $\mathcal{B} = \mbox{{\sf mod}-}B$ for some finite-dimensional $k$-algebras $A$ and $B$, then $\mbox{{\sf mod}-}(A \otimes_k B)$ is a tensor product of $\mathcal{A}$ and $\mathcal{B}$ with
\begin{equation*}
  \boxtimes: \mathcal{A} \times \mathcal{B} \to \mbox{{\sf mod}-}(A \otimes_k B),
  \quad (X, Y) \mapsto X \otimes_k Y
  \quad (X \in \mathcal{A}, Y \in \mathcal{B})
\end{equation*}
\cite[Proposition 5.3]{MR1106898}. The following lemma is obtained by using this realization of a tensor product of finite abelian categories:

\begin{lemma}[Deligne {\cite[Proposition 5.13]{MR1106898}}]
  \label{lem:Deligne-product}
  Let $\mathcal{A}$ and $\mathcal{B}$ be finite abelian categories over a field $k$. Then the following statements hold:
  \begin{enumerate}
  \item A tensor product $\mathcal{A} \boxtimes \mathcal{B}$ exists and is a finite abelian category over $k$.
  \item The functor $\boxtimes: \mathcal{A} \times \mathcal{B} \to \mathcal{A} \boxtimes \mathcal{B}$ is $k$-linear and exact in each variable.
  \item The functor $\LEX(\mathcal{A} \boxtimes \mathcal{B}, \mathcal{C}) \to \LEX_2(\mathcal{A}, \mathcal{B}; \mathcal{C})$ induced by $\boxtimes$ is an equivalence of categories for any $k$-linear abelian category $\mathcal{C}$.
  \item There is a natural isomorphism
    \begin{equation*}
      \Hom_{\mathcal{A} \boxtimes \mathcal{B}}(V \boxtimes W, X \boxtimes Y)
      \cong \Hom_{\mathcal{A}}(V, X) \otimes_k \Hom_{\mathcal{B}}(W, Y)
    \end{equation*}
    for $V, X \in \mathcal{A}$ and $W, Y \in \mathcal{B}$.
  \end{enumerate}
\end{lemma}

Now let $\mathcal{A}$ and $\mathcal{B}$ be finite abelian categories over $k$. We consider the functor
\begin{equation*}
  \Phi_1: \mathcal{A} \times \mathcal{B}^{\op} \times \mathcal{B} \to \mathcal{A},
  \quad (V, X, Y) \mapsto \Hom_{\mathcal{B}}(X, Y) \cdot V
  \quad (V \in \mathcal{A}, X, Y \in \mathcal{B}),
\end{equation*}
where ``$\cdot$'' is the canonical action of $\mbox{{\sf mod}-}k$ on $\mathcal{A}$ given in Example~\ref{ex:mod-cat-1}. By Part (3) of the above lemma, this functor induces a $k$-linear left exact functor
\begin{equation*}
  \Phi_2: \mathcal{A} \boxtimes \mathcal{B}^{\op} \boxtimes \mathcal{B} \to \mathcal{A},
  \quad V \boxtimes X \boxtimes Y
  \mapsto \Phi_1(V, X, Y).
\end{equation*}
By Part (2) of the above lemma, we have a functor
\begin{equation*}
  \Phi_3: \mathcal{A} \boxtimes \mathcal{B}^{\op} \to \LEX(\mathcal{B}, \mathcal{A}),
  \quad M \mapsto (V \mapsto \Phi_2(M \boxtimes V)).
\end{equation*}
For simplicity, we express the functor $\Phi_3$ obtained in this way as
\begin{equation}
  \label{eq:Lex-A-B-equiv-2}
  \mathcal{A} \boxtimes \mathcal{B}^{\op} \to \LEX(\mathcal{B}, \mathcal{A}),
  \quad V \boxtimes W \mapsto \Hom_{\mathcal{B}}(W, -) \cdot V
  \quad (V \in \mathcal{A}, W \in \mathcal{B})
\end{equation}

\begin{lemma}
  \label{lem:Lex-A-B-equiv-1}
  The functor \eqref{eq:Lex-A-B-equiv-2} is an equivalence.
\end{lemma}
\begin{proof}
  We may assume that $\mathcal{A} = \mbox{{\sf mod}-}A$ and $\mathcal{B} = \mbox{{\sf mod}-}B$ for some finite-dimensional algebras $A$ and $B$. By Lemma~\ref{lem:EW-thm} and the Yoneda lemma, we see that the following functor is an equivalence:
  \begin{equation*}
    (A\mbox{-{\sf mod}-}B)^{\op} \to \mathcal{L} := \LEX(\mathcal{A}, \mathcal{B}),  \quad M \mapsto \Hom_B(M, -) \quad (M \in A\mbox{-{\sf mod}-}B),
  \end{equation*}
  where $A\mbox{-{\sf mod}-}B$ is the category of finite-dimensional $A$-$B$-bimodules. In view of the above realization of a tensor product, we also have an equivalence
  \begin{equation*}
    \mathcal{A} \boxtimes \mathcal{B}^{\op} \to (A\mbox{-{\sf mod}-}B)^{\op},
    \quad V \boxtimes W \mapsto V^* \otimes_k W
    \quad (V \in \mathcal{A}, W \in \mathcal{B}),
  \end{equation*}
  where $A$ acts on $V^* := \Hom_k(V, k)$ by $a \cdot f = f(- \cdot a)$ ($a \in A, f \in V^*$). One can check that \eqref{eq:Lex-A-B-equiv-2} is obtained by composing these equivalences.
\end{proof}

The following description of a quasi-inverse of \eqref{eq:Lex-A-B-equiv-2} is important:

\begin{lemma}
  \label{lem:Lex-A-B-equiv-2}
  Notations are the same as in Lemma~\ref{lem:Lex-A-B-equiv-1}. For all $F \in \LEX(\mathcal{B}, \mathcal{A})$, a coend of the functor
  \begin{equation}
    \label{eq:Lex-A-B-equiv-3}
    \mathcal{B} \times \mathcal{B}^{\op}
    \to \mathcal{A} \boxtimes \mathcal{B}^{\op},
    \quad (X, Y) \mapsto F(X) \boxtimes Y
    \quad (X, Y \in \mathcal{B})
  \end{equation}
  exists. A quasi-inverse of \eqref{eq:Lex-A-B-equiv-2} is given by
  \begin{equation*}
    \LEX(\mathcal{B}, \mathcal{A}) \to \mathcal{A} \boxtimes \mathcal{B}^{\op},
    \quad F \mapsto \int^{X \in \mathcal{B}} F(X) \boxtimes X
    \quad (F \in \LEX(\mathcal{B}, \mathcal{A})).
  \end{equation*}
\end{lemma}
\begin{proof}
  For $F \in \LEX(\mathcal{B}, \mathcal{A})$, there are isomorphisms
  \begin{align*}
    \Hom_{\mathcal{A} \boxtimes \mathcal{B}^{\op}}(F(X) \boxtimes Y, V \boxtimes W)
    & \cong \Hom_{\mathcal{A}}(F(X), V) \otimes_k \Hom_{\mathcal{B}^{\op}}(Y, W) \\
    & \cong \Hom_{\mathcal{A}}(F(X), \Hom_{\mathcal{B}}(W, Y) \cdot V)
  \end{align*}
  natural in $V \in \mathcal{A}$ and $W, X, Y \in \mathcal{B}$ by Lemma~\ref{lem:Deligne-product} (4) and \eqref{eq:internal-Hom-tensor}. Since both sides are $k$-linear and left exact in the variables $V$ and $W$, we obtain
  \begin{equation*}
    \Hom_{\mathcal{A} \boxtimes \mathcal{B}^{\op}}(F(X) \boxtimes Y, M)
    \cong \Hom_{\mathcal{A}}(F(X), \Phi(M)(Y))
    \quad (M \in \mathcal{A} \boxtimes \mathcal{B}^{\op}),
  \end{equation*}
  where $\Phi$ is the equivalence given by~\eqref{eq:Lex-A-B-equiv-2}. Taking ends, we get
  \begin{equation*}
    \int_{X \in \mathcal{A}} \Hom_{\mathcal{A} \boxtimes \mathcal{B}^{\op}}(F(X) \boxtimes X, M)
    \cong \NAT(F, \Phi(M)).
  \end{equation*}
  Let $\overline{\Phi}$ be a quasi-inverse of $\Phi$. Since $\NAT(F, \Phi(-))$ is represented by $\overline{\Phi}(F)$, a coend of \eqref{eq:Lex-A-B-equiv-3} exists and is isomorphic to $\overline{\Phi}(F)$ by Lemma~\ref{lem:coend-existence}.
\end{proof}

Following Kerler and Lyubashenko \cite[\S5.1.3]{MR1862634}, a coend of $Q: \mathcal{A} \times \mathcal{A}^{\op} \to \mathcal{B}$ exists if $Q$ is $k$-linear and exact in each variable. Thus, in the case where $F$ is exact, the existence of a coend of~\eqref{eq:Lex-A-B-equiv-3} follows from their result. Theorem~\ref{thm:FTC-coend-exist} below also follows from their result in such a case.

\begin{theorem}
  \label{thm:FTC-coend-exist}
  Let $\mathcal{C}$ be a finite tensor category over a field $k$. Then coends
  \begin{equation*}
    \int^{X \in \mathcal{C}} F(X^*) \boxtimes X
    \text{\quad and \quad} \int^{X \in \mathcal{C}} F(X^*) \otimes X
  \end{equation*}
  exist for all $F \in \LEX(\mathcal{C}, \mathcal{C})$.
\end{theorem}
\begin{proof}
  Note that $F(-^*): \mathcal{C}^{\op} \to \mathcal{C}$ is $k$-linear left exact if $F \in \LEX(\mathcal{C}, \mathcal{C})$. Hence, applying the above lemma to $F(-^*)$, we see that the first coend exists. The second coend is obtained by applying the right exact functor $X \boxtimes Y \mapsto X \otimes Y$ to the first coend.
\end{proof}

\begin{remark}
  \label{rem:FTC-coend-remark}
  For $F \in \LEX(\mathcal{C}, \mathcal{C})$, there is an isomorphism
  \begin{equation*}
    \int^{X \in \mathcal{C}} F(X^*) \boxtimes X \cong \int^{X \in \mathcal{C}} F(X) \boxtimes {}^* \! X.
  \end{equation*}
  Indeed, for every object $C \in \mathcal{C} \boxtimes \mathcal{C}^{\op}$, the map
  \begin{equation*}
    \text{\sc Dinat}(F(-) \boxtimes {}^*(-), C) \to \text{\sc Dinat}(F(-^*) \boxtimes (-), C),
    \quad \{ i_V \}_{V \in \mathcal{C}} \mapsto \{ i_{V^*} \}_{V \in \mathcal{C}}
  \end{equation*}
  is a bijection, where $\text{\sc Dinat}(P, Q)$ means the set of dinatural transformations from $P$ to $Q$. Similarly, there is an isomorphism
  \begin{equation*}
    \int^{X \in \mathcal{C}} F(X^*) \otimes X \cong \int^{X \in \mathcal{C}} F(X) \otimes {}^* \! X.
  \end{equation*}
\end{remark}

\subsection{The center of a finite tensor category}

Applying Theorem~\ref{thm:FTC-coend-exist} to $F = (-) \otimes V$, we see that the coend in the right-hand side of \eqref{eq:Hopf-monad-Z} always exists in a finite tensor category. As an application of this result, we prove:

\begin{theorem}
  \label{thm:FTC-center-FTC}
  The center of a finite tensor category is a finite tensor category.
\end{theorem}
\begin{proof}
  Let $\mathcal{C}$ be a finite tensor category over a field $k$. As we have remarked, the central Hopf monad $Z$ on $\mathcal{C}$ exists and therefore we can identify $\mathcal{Z}(\mathcal{C})$ as the category ${}_Z \mathcal{C}$ of $Z$-modules. Set $Z^!(V) = {}^*Z(V^*)$ for $V \in \mathcal{C}$. By Remark~\ref{rem:FTC-coend-remark}, we have
  \begin{align*}
    \Hom_{\mathcal{C}}(W, Z^!(V))
    & \cong \Hom_{\mathcal{C}}(Z(V^*), W^*) \\
    & \cong \textstyle \int_{X \in \mathcal{C}} \Hom_{\mathcal{C}}(X^* \otimes V^* \otimes X, W^*) \\
    & \cong \textstyle \int_{X \in \mathcal{C}} \Hom_{\mathcal{C}}(X \otimes W \otimes {}^* X, V) \\
    & \cong \Hom_{\mathcal{C}}(Z(W), V)
  \end{align*}
  for all $V, W \in \mathcal{C}$. Hence the functor $Z^!$ is a right adjoint of $Z$ (this result is a special case of \cite[Corollary~3.12]{MR2355605}). Now the result follows from Lemma~\ref{lem:finite-T-mod}.
\end{proof}

\begin{remark}
  \label{rem:non-perfect-k}
  Let $\mathcal{C}$ and $\mathcal{D}$ be finite tensor categories over a field $k$. Then $\mathcal{C} \boxtimes \mathcal{D}$ is a $k$-linear monoidal category with tensor product determined by
  \begin{equation*}
    (V \boxtimes W) \otimes (X \boxtimes Y) = (V \otimes X) \boxtimes (W \otimes Y)
    \quad (V, X \in \mathcal{C}, W, Y \in \mathcal{D})
  \end{equation*}
  and unit $\unitobj \boxtimes \unitobj$. Following Deligne \cite[Proposition 5.17]{MR1106898}, $\mathcal{C} \boxtimes \mathcal{D}$ is a finite tensor category provided that $k$ is a perfect field.

  Theorem~\ref{thm:FTC-center-FTC} is proved in \cite{MR2119143} under the assumption that $k$ is algebraically closed and $\unitobj \in \mathcal{C}$ is simple. Their proof relies on the fact that $\mathcal{C} \boxtimes \mathcal{C}^{\rev}$ is a finite tensor category and thus cannot be applied to our case.
\end{remark}

\section{Characterizations of unimodularity}
\label{sec:main-thm}

\subsection{General assumptions}

Let $\mathcal{C}$ be a finite tensor category over a field $k$. Then $\mathcal{C} \boxtimes \mathcal{C}^{\rev}$ is a monoidal category with tensor product
\begin{equation*}
  (V \boxtimes W) \otimes (X \boxtimes Y) = (V \otimes X) \boxtimes (Y \otimes W)
  \quad (V, W, X, Y \in \mathcal{C})
\end{equation*}
and unit $\unitobj \boxtimes \unitobj$. Throughout this section, we assume that
\begin{equation}
  \label{eq:C-env-assume}
  \text{$\mathcal{C}^{\env} := (\mathcal{C} \boxtimes \mathcal{C}^{\rev}, \otimes, \unitobj \boxtimes \unitobj)$ is a finite tensor category},
\end{equation}
which holds if the field $k$ is perfect (see Remark~\ref{rem:non-perfect-k}). We note that \eqref{eq:C-env-assume} holds also in the case where $\mathcal{C}$ is the representation category of a finite-dimensional (quasi-)Hopf algebra.

\subsection{The definition of unimodularity}

Following \cite{MR2097289}, we recall the definition of the distinguished invertible object and the unimodularity of a finite tensor category. The category $\mathcal{C}$ is a finite $\mathcal{C}^{\env}$-module category with the action determined by
\begin{equation}
  \label{eq:C-env-action-1}
  (V \boxtimes W) \ogreaterthan X = V \otimes X \otimes W
  \quad (V, W, X \in \mathcal{C}).
\end{equation}
Now we set $A = \iHom(\unitobj, \unitobj) \in \mathcal{C}^{\env}$. By~\eqref{eq:internal-Hom-tensor}, we have
\begin{equation}
  \label{eq:C-env-action-2}
  \iHom(V \otimes W, X \otimes Y) \cong (X \boxtimes Y) \otimes A \otimes (V \boxtimes W)^*
  \quad (V,W,X,Y \in \mathcal{C}).
\end{equation}
Theorem~\ref{thm:mod-cat-comparison} implies that the functor
\begin{equation}
  \label{eq:Hopf-module-equiv}
  \mathcal{C} \to (\mathcal{C}^{\env})_A,
  \quad V \mapsto \iHom(\unitobj, V) \cong (V \boxtimes \unitobj) \otimes A_A
  \quad (V \in \mathcal{C})
\end{equation}
is an equivalence of $\mathcal{C}^{\env}$-module categories. In view of this equivalence, there exists an object $D \in \mathcal{C}$, which is unique up to isomorphism, such that
\begin{equation}
  \label{eq:d-inv-obj-def}
  (D \boxtimes \unitobj) \otimes A_A \cong ({}_A A)^*.
\end{equation}

\begin{definition}[\cite{MR2097289}]
   The object $D$ is called the {\em distinguished invertible object} of $\mathcal{C}$, and the finite tensor category $\mathcal{C}$ is said to be {\em unimodular} if $D \cong \unitobj$.
\end{definition}

As its name suggests, $D$ is an invertible object, {\it i.e.}, the evaluation $\eval_D$ and the coevaluation $\coev_D$ are isomorphisms. In \cite{MR2097289}, the invertibility is proved under the assumption that $k$ is algebraically closed and $\unitobj \in \mathcal{C}$ is simple. Their proof relies on the theory of the Frobenius-Perron dimension, and thus cannot be applied to our case. To prove the invertibility of $D$, we first note:

\begin{lemma}
  \label{lem:int-Hom-dual}
  $\iHom(V, X)^* \cong \iHom(X, V^{**} \otimes D)$.
\end{lemma}
\begin{proof}
  By \eqref{eq:C-env-action-2} with $W = Y = \unitobj$ and~\eqref{eq:d-inv-obj-def}, we compute
  \begin{align*}
    \iHom(V, X)^*
    & \cong ((X \boxtimes \unitobj) \otimes A \otimes (V \boxtimes \unitobj)^*)^* \\
    & \cong (V^{**} \boxtimes \unitobj) \otimes (D \boxtimes \unitobj) \otimes A \otimes (X \boxtimes \unitobj)^* \\
    & \cong \iHom(X, V^{**} \otimes D). \qedhere
  \end{align*}
\end{proof}

\begin{lemma}
  $D$ is invertible.
\end{lemma}
\begin{proof}
  By the previous lemma, we have natural isomorphisms
  \begin{equation*}
    \iHom({}^{**}V, \unitobj)^{**} \cong \iHom(\unitobj, V \otimes D)^* \cong \iHom(V \otimes D, D)
  \end{equation*}
  for $V \in \mathcal{C}$. Thus we compute:
  \begin{align*}
    \Hom_{\mathcal{C}}(V, D \otimes D^*)
    & \cong \Hom_{\mathcal{C}}(V \otimes D, D) \\
    & \cong \Hom_{\mathcal{C}^{\env}}(\unitobj \boxtimes \unitobj, \iHom(V \otimes D, D)) \\
    & \cong \Hom_{\mathcal{C}^{\env}}(\unitobj \boxtimes \unitobj, \iHom({}^{**}V, \unitobj)^{**}) \\
    & \cong \Hom_{\mathcal{C}^{\env}}(\unitobj \boxtimes \unitobj, \iHom({}^{**}V, \unitobj)) \\
    & \cong \Hom_{\mathcal{C}}({}^{**}V, \unitobj) \cong \Hom_{\mathcal{C}}(V, \unitobj).
  \end{align*}
  By the Yoneda lemma, $D \otimes D^* \cong \unitobj$. This implies the invertibility of $D$, since $\mathcal{C}$ is a multi-ring category in the sense of \cite{EGNO-Lect}; see \cite[\S1.15]{EGNO-Lect}.
\end{proof}

\subsection{The algebra $A$ as a coend}
\label{subsec:alg-A-coend}

The first step for the proof of our main theorem is to describe the algebra $A$ as a coend of a certain functor. Note that the left duality functor is an equivalence $(-)^*: \mathcal{C}^{\rev} \to \mathcal{C}^{\op}$ with quasi-inverse ${}^*(-)$. Hence, by Lemmas~\ref{lem:Lex-A-B-equiv-1} and~\ref{lem:Lex-A-B-equiv-2}, the functor
\begin{equation}
  \label{eq:Lex-C-equiv-1}
  \Phi: \mathcal{C}^{\env} \to \LEX(\mathcal{C})
  \quad (:= \LEX(\mathcal{C}, \mathcal{C})),
  \quad V \boxtimes W \mapsto \Hom_{\mathcal{C}}(W^*, -) \cdot V
\end{equation}
is an equivalence of categories with quasi-inverse given by
\begin{equation}
  \label{eq:Lex-C-equiv-2}
  \overline{\Phi}: \LEX(\mathcal{C}) \to \mathcal{C}^{\env},
  \quad F \mapsto \int^{X \in \mathcal{C}} F(X) \boxtimes {}^* \! X.
\end{equation}
Recall that the category $\mathcal{C}$ is a finite $\mathcal{C}^{\env}$-module category by~\eqref{eq:C-env-action-1}. The internal Hom functor for $\mathcal{C}$ is given as follows:

\begin{lemma}
  \label{lem:internal-Hom-coend}
  $\iHom(V, W) = \overline{\Phi}(W \otimes (-) \otimes V^*)$.
\end{lemma}
\begin{proof}
  For $V, W \in \mathcal{C}$ and $F \in \LEX(\mathcal{C})$, we compute
  \begin{align*}
    \Hom_{\mathcal{C}^{\env}}(\overline{\Phi}(F), \overline{\Phi}(W \otimes (-) \otimes V^*))
    & \cong \NAT(F, W \otimes (-) \otimes V^*) \\
    & \cong \textstyle \int_{X \in \mathcal{C}} \Hom_{\mathcal{C}}(F(X), W \otimes X \otimes V^*) \\
    & \cong \textstyle \int_{X \in \mathcal{C}} \Hom_{\mathcal{C}}((F(X) \boxtimes {}^* X) \ogreaterthan V, W) \\
    & \cong \Hom_{\mathcal{C}}(\overline{\Phi}(F) \ogreaterthan V, W).
  \end{align*}
  Since $\overline{\Phi}$ is an equivalence, the claim follows from the above computation.
\end{proof}

Let $F \in \LEX(\mathcal{C})$ and $V, W \in \mathcal{C}$. We pay attention to the bijection
\begin{equation*}
  \NAT(F, W \otimes (-) \otimes V^*) \cong \Hom_{\mathcal{C}}(\overline{\Phi}(F) \ogreaterthan V, W)
\end{equation*}
in the proof of Lemma~\ref{lem:internal-Hom-coend}. The morphism $f: \overline{\Phi}(F) \ogreaterthan V \to W$ corresponding to a natural transformation $\alpha: F \to W \otimes (-) \otimes V^*$ via the above bijection is uniquely determined by the property that the diagram
\begin{equation*}
  \xymatrix{
    (F(X) \boxtimes {}^* \! X) \ogreaterthan V \ar@{=}[d]
    \ar[rr]^(.55){j'_{F}(X) \ogreaterthan V}
    & & \overline{\Phi}(F) \ogreaterthan V \ar[r]^(.6){f}
    & W \\
    F(X) \otimes V \otimes {}^* \! X \ar[rrr]_(.45){\alpha_X \otimes V \otimes {}^* \! X}
    & & & W \otimes X \otimes V^* \otimes V \otimes {}^* \! X \ar[u]_{W \otimes \eval_{V \otimes {}^* \! X}}
  }
\end{equation*}
commutes for all $X \in \mathcal{C}$, where $j'_{F}(X): F(X) \boxtimes {}^* \! X \to \overline{\Phi}(F)$ is the component of the
universal dinatural transformation. In particular, the evaluation $\ieval_{V,W}$ for $V, W \in \mathcal{C}$ is the morphism making the diagram
\begin{equation}
  \label{eq:internal-Hom-coend-eval}
  \xymatrix{
    ((W \otimes X \otimes V^*) \boxtimes {}^* X) \ogreaterthan V
    \ar@{=}[d] \ar[rrr]^(.6){j''_{V,W}(X) \ogreaterthan V}
    & & & \iHom(V, W) \ogreaterthan V
    \ar[d]^{\ieval_{V,W}} \\
    W \otimes X \otimes V^* \otimes V \otimes {}^* X
    \ar[rrr]_(.6){W \otimes \eval_{V \otimes {}^* X}}
    & & & W
  }
\end{equation}
commutes for all $X \in \mathcal{C}$, where $j''_{V,W} = j'_{F}$ with $F = W \otimes (-) \otimes V^*$. Now we set $j = j_{\unitobj,\unitobj}''$. The algebra structure of $A = \iHom(\unitobj, \unitobj)$ is described by using the dinatural transformation $j$ as follows:

\begin{lemma}
  \label{lem:alg-A-coend}
  With the above notation, the multiplication $m: A \otimes A \to A$ is a unique morphism such that the diagram
  \begin{equation}
    \label{eq:alg-A-mult}
    \xymatrix{
      A \otimes A \ar[d]_{m}
      & & & \ar[lll]_(.6){j(X) \otimes j(Y)} \ar@{=}[d]
      (X \boxtimes {}^* \! X) \otimes (Y \boxtimes {}^* Y) \\
      A
      & & & \ar[lll]^(.6){j(X \otimes Y)} (X \otimes Y) \boxtimes ({}^* Y \otimes {}^* \! X)
    }
  \end{equation}
  commutes for all $X, Y \in \mathcal{C}$. The unit $u: \unitobj \boxtimes \unitobj \to A$ is given by $u = j(\unitobj)$.
\end{lemma}
\begin{proof}
  It is easy to see that the unit of $A$ is given as stated. For $X, Y \in \mathcal{C}$, we have a commutative diagram
  \begin{equation*}
    \xymatrix{
      ((X \boxtimes {}^* \! X) \otimes (Y \boxtimes {}^* Y)) \ogreaterthan \unitobj
      \ar[rrr]^(.6){(j(X) \otimes j(Y)) \ogreaterthan \unitobj} \ar[d]_{\cong}
      & & & (A \otimes A) \ogreaterthan \unitobj \ar[d]_{\cong}
      \ar[rr]^(.6){m \ogreaterthan \unitobj}
      & & A \ogreaterthan \unitobj \ar[dd]_{\ieval_{\unitobj,\unitobj}} \\
      (X \boxtimes {}^* \! X) \ogreaterthan ((Y \boxtimes {}^* Y) \ogreaterthan \unitobj)
      \ar[rrr]^(.6){j(X) \ogreaterthan (j(Y) \ogreaterthan \unitobj)}
      \ar[d]_{(X \boxtimes {}^* \! X) \ogreaterthan \eval_{{}^*Y}}
      & & & A \ogreaterthan (A \ogreaterthan \unitobj)
      \ar[d]_{A \ogreaterthan \ieval_{\unitobj,\unitobj}} \\
      (X \boxtimes {}^* X) \ogreaterthan \unitobj
      \ar[rrr]_(.6){j(X) \ogreaterthan \unitobj}
      & & & A \ogreaterthan \unitobj \ar[rr]_{\ieval_{\unitobj,\unitobj}}
      & & \unitobj
    }
  \end{equation*}
  by~\eqref{eq:internal-Hom-coend-eval} and the definition of $m$. Again by~\eqref{eq:internal-Hom-coend-eval}, the composition along the bottom row is $\eval_{{}^*\!X}$. Hence we obtain:
  \begin{align*}
    \ieval_{\unitobj,\unitobj} \circ (m \ogreaterthan \unitobj) \circ ((j(X) \otimes j(Y)) \ogreaterthan \unitobj)
    & = \eval_{{}^* \! X} \circ (\id_{X \boxtimes {}^* \! X} \ogreaterthan \eval_{{}^*Y}) \\
    & = \eval_{{}^* \! X} \circ (\id_{X} \otimes \eval_{{}^*Y} \otimes \id_{{}^* \! X}) \\
    & = \eval_{{}^* Y \otimes {}^* X} \\
    & = \ieval_{\unitobj,\unitobj} \circ (j(X \otimes Y) \ogreaterthan \unitobj).
  \end{align*}
  Since the map $\Hom_{\mathcal{C}^{\env}}(M, A) \to \Hom_{\mathcal{C}}(M \ogreaterthan \unitobj, \unitobj)$ given by $f \mapsto \ieval_{\unitobj,\unitobj} \circ (f \ogreaterthan \unitobj)$ is bijective, the commutativity of~\eqref{eq:alg-A-mult} follows.
\end{proof}

\subsection{The algebra $A$ and the central Hopf monad}

For $V, X \in \mathcal{C}$, we set
\begin{equation*}
  Z(V) = A \ogreaterthan V
  \text{\quad and \quad}
  i_V(X) = j(X^*) \ogreaterthan V:
  X^* \otimes V \otimes X \to Z(V),
\end{equation*}
where $A$ and $j$ are as before. Since $A$ is an algebra in $\mathcal{C}^{\env}$, the functor $Z$ has a structure of a monad. More precisely, the multiplication of $Z$ is given by
\begin{equation*}
  \mu_V: Z^2(V) = A \ogreaterthan (A \ogreaterthan V)
  = (A \otimes A) \ogreaterthan V
  \xrightarrow{\ m \ogreaterthan V \ } A \ogreaterthan V = Z(V)
  \quad (V \in \mathcal{C})
\end{equation*}
and the unit of $Z$ is given by
\begin{equation*}
  \eta_V: V = \unitobj \ogreaterthan V
  \xrightarrow{\ u \ogreaterthan V \ } A \ogreaterthan V = Z(V)
  \quad (V \in \mathcal{C}).
\end{equation*}
Note that $\{ i_V(X) \}_{X \in \mathcal{C}}$ is a coend of $(X, Y) \mapsto X^* \otimes V \otimes Y$. By Lemma~\ref{lem:internal-Hom-coend}, one can check that the unit $\eta$ is given by $\eta_V = i_V(\unitobj)$ for $V \in \mathcal{C}$ and the multiplication $\mu$ is determined by the same formula as \eqref{eq:Hopf-monad-Z-mult}. In conclusion, the monad $Z$ under consideration is precisely the central Hopf monad on $\mathcal{C}$.

Let $K: \mathcal{C} \to \mathcal{H} := (\mathcal{C}^{\env})_A$ be the equivalence given by \eqref{eq:Hopf-module-equiv}. Note that the functor $T = A \otimes (-)$ defines a monad on $\mathcal{H}$ such that ${}_T \mathcal{H} = {}_{A}(\mathcal{C}^{\env})_A$. Since $K$ is in fact an equivalence of $\mathcal{C}^{\env}$-module categories, it induces an equivalence between ${}_Z\mathcal{C}$ and ${}_T \mathcal{H}$. More precisely, if $M$ is a $Z$-module with action $\rho$, then $K(M) \in \mathcal{H}$ is an $A$-bimodule with the left action given by
\begin{equation*}
  A \otimes K(M) \xrightarrow{\quad \cong \quad} K(A \ogreaterthan M) = K(Z(M)) \xrightarrow{\ K(\rho) \ } K(M)
\end{equation*}
and this construction gives rise to an equivalence of categories
\begin{equation}
  \label{eq:ZC-and-A-bimod-eq}
  \widetilde{K}: {}_Z \mathcal{C} \xrightarrow{\ \approx \ } {}_T \mathcal{H} = {}_A (\mathcal{C}^{\env})_A,
  \quad M \mapsto K(M)
  \quad (M \in {}_Z \mathcal{C}).
\end{equation}
Recall from \S\ref{subsec:Hopf-monads} that ${}_Z \mathcal{C}$ can be identified with $\mathcal{Z}(\mathcal{C})$. By the definition of $\widetilde{K}$, it is obvious that the following diagram commutes:
\begin{equation}
  \label{eq:ZC-and-A-bimod}
  \xymatrix{
    \mathcal{Z}(\mathcal{C}) \ar[rr]^{\widetilde{K}}
    \ar[d]_{U}
    & & {}_T \mathcal{H} \ar@{=}[r] & {}_A (\mathcal{C}^{\env})_A
    \ar[d]^{F_A} \\
    \mathcal{C} \ar[rr]_{K}
    & & \mathcal{H} \ar@{=}[r] & (\mathcal{C}^{\env})_A,
  }
\end{equation}
where $U$ and $F_A$ are forgetful functors.

\begin{remark}
  Etingof and Ostrik \cite[Corollary 3.35]{MR2119143} showed that ${}_A (\mathcal{C}^{\env})_A$ is equivalent to $\mathcal{Z}(\mathcal{C})$. However, since they did not give an equivalence in an explicit way, it is not clear that there exists a commutative diagram like \eqref{eq:ZC-and-A-bimod}. In this paper, we give an equivalence ${}_A (\mathcal{C}^{\env})_A \approx \mathcal{Z}(\mathcal{C})$ in a somewhat explicit way by investigating the relation between the algebra $A$ and the central Hopf monad. The commutativity of \eqref{eq:ZC-and-A-bimod} is obvious from our point of view.
\end{remark}

\subsection{Characterizations of unimodularity}

In this subsection, we prove the main theorem of this paper. Recall our assumption that $\mathcal{C}$ is a finite tensor category with property \eqref{eq:C-env-assume}. Let $L$ and $R$ be a left adjoint and a right adjoint of the forgetful functor $U: \mathcal{Z}(\mathcal{C}) \to \mathcal{C}$, and let $D \in \mathcal{C}$ be the distinguished invertible object. The key observation is the following lemma:

\begin{lemma}
  \label{lem:induction-L-R-1}
  There are natural isomorphisms
  \begin{equation*}
    L(D \otimes -) \cong R \cong L(- \otimes D)
    \quad \text{and} \quad
    R(D^* \otimes -) \cong L \cong R(- \otimes D^*).
  \end{equation*}    
\end{lemma}
\begin{proof}
  Let $\widetilde{K}$ be the equivalence given by \eqref{eq:ZC-and-A-bimod-eq}. By Lemma \ref{lem:adj-restriction} and the commutativity of the diagram \eqref{eq:ZC-and-A-bimod}, we have
  \begin{equation*}
    \widetilde{K} L(V)
    \cong {}_A A \otimes K(V)
    \cong {}_A A \otimes (V \boxtimes \unitobj) \otimes A_A.
  \end{equation*}
  Note that ${}^*(A_A) \cong {}_A A \otimes (D \boxtimes \unitobj)$ by \eqref{eq:d-inv-obj-def}. Again by Lemma \ref{lem:adj-restriction}, we have
  \begin{align*}
      \widetilde{K} R(V)
      & \cong {}^*(A_A) \otimes (V \boxtimes \unitobj) \otimes A_A \\
      & \cong {}_A A \otimes (D \boxtimes \unitobj) \otimes (V \boxtimes \unitobj) \otimes A_A
      \cong \smash{\widetilde{K}} L(D \otimes V)
  \end{align*}
  for $V \in \mathcal{C}$. Hence we obtain the first natural isomorphism. The third one is obtained from the first one and the fact that $D$ is invertible as follows:
  \begin{equation*}
    R(D^* \otimes V) \cong L(D \otimes D^* \otimes V) \cong L(V).
  \end{equation*}
  For a functor $T$ between rigid monoidal categories, we set $T^!(X) = {}^*T(X^*)$. Recall from \S\ref{subsec:colax-lax-adj} that $R \cong L^!$ and $L \cong R^!$. The second one is obtained from this fact and the third natural isomorphism as follows:
  \begin{equation*}
    R(V) \cong {}^* \! L(V^*) \cong {}^* \! R(D^* \otimes V^*) \cong {}^* \! R((V \otimes D)^*) \cong L(V \otimes D).
  \end{equation*}
  The last one is obtained from the second one and the invertibility of $D$.
\end{proof}

The above lemma implies many relations between $U$, $L$, $R$ and $D$. Here we prove the following lemma, which describes a left adjoint of $L$ and a right adjoint of $R$ in terms of $U$ and $D$.

\begin{corollary}
  \label{cor:induction-L-R-2}
  There are natural isomorphisms
  \begin{gather*}
    \Hom_{\mathcal{C}}(D \otimes U(X), V)
    \cong \Hom_{\mathcal{Z}(\mathcal{C})}(X, L(V))
    \cong \Hom_{\mathcal{C}}(U(X) \otimes D, V), \\
    \Hom_{\mathcal{C}}(V, D^* \otimes U(X))
    \cong \Hom_{\mathcal{Z}(\mathcal{C})}(R(V), X)
    \cong \Hom_{\mathcal{C}}(V, U(X) \otimes D^*).
  \end{gather*}
\end{corollary}
\begin{proof}
  We only give the first natural isomorphism, since the others are obtained in a similar way. By Lemma~\ref{lem:induction-L-R-1} and the invertibility of $D$, we have:
  \begin{align*}
    \Hom_{\mathcal{Z}(\mathcal{C})}(X, L(V))
    & \cong \Hom_{\mathcal{Z}(\mathcal{C})}(X, R(D^* \otimes V)) \\
    & \cong \Hom_{\mathcal{Z}(\mathcal{C})}(U(X), D^* \otimes V) \\
    & \cong \Hom_{\mathcal{Z}(\mathcal{C})}(D \otimes U(X), V). \qedhere
  \end{align*}
\end{proof}

\begin{corollary}
  \label{cor:induction-L-R-exact}
  $L$ and $R$ are exact.
\end{corollary}

Now we prove our main theorem:

\begin{theorem}
  \label{thm:main-theorem}
  With the above notation, the following assertions are equivalent:
  \begin{enumerate}
  \item $\mathcal{C}$ is unimodular.
  \item $U$ is a Frobenius functor, {\it i.e.}, $L \cong R$.
  \item There is a natural isomorphism $L(V^*) \cong L(V)^*$.
  \item There is a natural isomorphism $R(V^*) \cong R(V)^*$.
  \end{enumerate}
  Moreover, if the unit object $\unitobj \in \mathcal{C}$ is simple, the above assertions are equivalent to each of the following assertions:
  \begin{enumerate}
    \setcounter{enumi}{4}
  \item $L(\unitobj) \cong L(\unitobj)^*$.
  \item $R(\unitobj) \cong R(\unitobj)^*$.
  \item $\Hom_{\mathcal{Z}(\mathcal{C})}(\unitobj, L(\unitobj)) \ne 0$.
  \item $\Hom_{\mathcal{Z}(\mathcal{C})}(R(\unitobj), \unitobj) \ne 0$.
  \end{enumerate}
\end{theorem}
\begin{proof}
  It is obvious from Lemma~\ref{lem:induction-L-R-1} that (1) implies (2). Using the isomorphisms $R \cong L^!$ and $L \cong R^!$, we easily see that (2), (3) and (4) are equivalent. We show that (2) implies (1). If (2) holds, then
  \begin{equation*}
    \Hom_{\mathcal{C}}(D, V)
    \cong \Hom_{\mathcal{Z}(\mathcal{C})}(\unitobj, L(V))
    \cong \Hom_{\mathcal{Z}(\mathcal{C})}(\unitobj, R(V))
    \cong \Hom_{\mathcal{C}}(\unitobj, V)
  \end{equation*}
  by Corollary~\ref{cor:induction-L-R-2} with $X = \unitobj$. Thus $D \cong \unitobj$, {\it i.e.}, $\mathcal{C}$ is unimodular. Hence we have showed that the assertions (1), (2), (3) and (4) are equivalent.

  It is obvious that (3) implies (5) (without the assumption that $\unitobj \in \mathcal{C}$ is a simple object). If (5) holds, then we have
  \begin{equation*}
    \Hom_{\mathcal{Z}(\mathcal{C})}(\unitobj, L(\unitobj))
    \cong \Hom_{\mathcal{Z}(\mathcal{C})}(\unitobj, L(\unitobj)^*)
    \cong \Hom_{\mathcal{Z}(\mathcal{C})}(L(\unitobj), \unitobj)
    \cong \Hom_{\mathcal{C}}(\unitobj, \unitobj)
  \end{equation*}
  and therefore (7) holds (again without the assumption on $\unitobj$). We prove (4) $\Rightarrow$ (6) and (6) $\Rightarrow$ (8) in a similar way. Moreover, since
  \begin{equation*}
    \Hom_{\mathcal{Z}(\mathcal{C})}(\unitobj, L(\unitobj))
    \cong \Hom_{\mathcal{Z}(\mathcal{C})}(L(\unitobj)^*, \unitobj)
    \cong \Hom_{\mathcal{Z}(\mathcal{C})}(R(\unitobj), \unitobj),
  \end{equation*}
  the assertions (7) and (8) are equivalent.

  Now we suppose that $\unitobj \in \mathcal{C}$ is simple. To complete the proof, it is sufficient to show that (7) implies (1). If (7) holds, then we have
  \begin{equation*}
    \Hom_{\mathcal{C}}(D, \unitobj)
    \cong \Hom_{\mathcal{Z}(\mathcal{C})}(\unitobj, L(\unitobj)) \ne 0
  \end{equation*}
  by Corollary~\ref{cor:induction-L-R-2}. Since $\unitobj \in \mathcal{C}$ is assumed to be simple, every invertible object of $\mathcal{C}$ is also simple. Thus, by Schur's lemma, we have $D \cong \unitobj$, {\it i.e.}, (1) holds.
\end{proof}

\section{Applications, I. Further results on the unimodularity}
\label{sec:applications-1}

\subsection{General assumptions}
\label{subsec:app-1-nota}

In this section, we apply our techniques to investigate further properties of the distinguished invertible object. As before, $\mathcal{C}$ is a finite tensor category over a field $k$ with property \eqref{eq:C-env-assume}, $D \in \mathcal{C}$ is the distinguished invertible object, $U: \mathcal{Z}(\mathcal{C}) \to \mathcal{C}$ is the forgetful functor, and $L$ and $R$ are a left adjoint and a right adjoint of $U$, respectively.

\subsection{The Radford $S^4$-formula}
\label{subsec:radford-s4}

Let $\mathcal{S}$ denote the left duality functor on $\mathcal{C}$, and let $\mathcal{I}_D: \mathcal{C} \to \mathcal{C}$ denote the monoidal functor defined by $\mathcal{I}_D(X) = D \otimes X \otimes D^*$ ($X \in \mathcal{C}$). One of the main results of \cite{MR2097289} is that there exists an isomorphism
\begin{equation}
  \label{eq:Radford-S4}
  \mathcal{S}^4 \cong \mathcal{I}_D
\end{equation}
of monoidal functors (the Radford $S^4$-formula). Here we explain how this formula looks like from the argument in the previous section (see \cite{2013arXiv1312.7188D} and \cite{2014arXiv1412.0211S} for other approaches to the Radford $S^4$-formula).

We have used the equivalence $\Phi: \mathcal{C}^{\env} \to \LEX(\mathcal{C})$ given by~\eqref{eq:Lex-C-equiv-1} to prove our main theorem. As before, we denote its quasi-inverse by $\overline{\Phi}$. We compute
\begin{align*}
  \Phi(\overline{\Phi}(F) \otimes \overline{\Phi}(G))
  & = \textstyle \int^{X,Y \in \mathcal{C}} \Phi((F(X) \boxtimes {}^*X) \otimes (G(X) \boxtimes {}^*Y)) \\
  & = \textstyle \int^{X,Y \in \mathcal{C}} \Phi((F(X) \otimes G(Y)) \boxtimes {}^* (X \otimes Y)) \\
  & = \textstyle \int^{X,Y \in \mathcal{C}} \Hom_{\mathcal{C}}(X \otimes Y, -) \cdot (F(X) \otimes G(Y)).
\end{align*}
Thus, $F \star G := \Phi(\overline{\Phi}(F) \otimes \overline{\Phi}(G))$ is the {\em Day convolution} \cite{DayPhDThesis} of $F$ and $G$, and the equivalence $\Phi$ is in fact a monoidal equivalence
\begin{equation*}
  \Phi: (\mathcal{C}^{\env}, \otimes, \unitobj \boxtimes \unitobj) \to (\LEX(\mathcal{C}), \star, J),
\end{equation*}
where $J = \Hom_{\mathcal{C}}(\unitobj, -) \cdot \unitobj$. In particular, $\Phi$ sends an algebra in $\mathcal{C}^{\env}$ to an algebra in $\LEX(\mathcal{C})$ with respect to the Day convolution, {\it i.e.}, a $k$-linear left exact monoidal endofunctor on $\mathcal{C}$ (see \cite[Example 3.2.2]{DayPhDThesis}).

Now let $\iHom$ denote the internal Hom functor for the $\mathcal{C}^{\env}$-module category $\mathcal{C}$, and let $A = \iHom(\unitobj, \unitobj)$ be the algebra in $\mathcal{C}^{\env}$ used to define $D$. By Lemma~\ref{lem:Nakayama-iso} and the definition of $D$, we obtain an isomorphism
\begin{equation}
  \label{eq:alg-A-double-dual}
  A^{**} \cong A^D := D \otimes A \otimes D^*
\end{equation}
of algebras in $\mathcal{C}^{\env}$. Since $(-)^{**}: \mathcal{C}^{\env} \to \mathcal{C}^{\env}$ is an equivalence, we have
\begin{equation*}
  A^{**}
  \cong \int^{X \in \mathcal{C}} (X \boxtimes {}^*X)^{**}
  \cong \int^{X} X^{**} \boxtimes {}^{***}X
  \cong \int^{X} X^{****} \boxtimes {}^{*}X
  \cong \overline{\Phi}(\mathcal{S}^4)
\end{equation*}
({\it cf}. Remark~\ref{rem:FTC-coend-remark}). We also have $A^D \cong \overline{\Phi}(\mathcal{I}_D)$ by Lemma~\ref{lem:internal-Hom-coend}. One can check that these isomorphisms are in fact isomorphisms of algebras in $\mathcal{C}^{\env}$.  Thus, to prove the Radford $S^4$-formula \eqref{eq:Radford-S4}, we only have to apply $\Phi$ to \eqref{eq:alg-A-double-dual}.

\subsection{Faithfulness of adjoints}

We have showed that $L$ and $R$ are exact functors (Corollary~\ref{cor:induction-L-R-exact}). Here we discuss the faithfulness of $L$ and $R$. The following lemma, which is possibly well-known, is useful for our purpose:

\begin{lemma}
  \label{lem:faithful-exact}
  For an exact functor $F: \mathcal{A} \to \mathcal{B}$ between abelian categories, the following assertions are equivalent:
  \begin{enumerate}
  \item $F$ is faithful.
  \item $F$ reflects isomorphisms.
  \item $F$ reflects zero objects.
  \end{enumerate}
\end{lemma}
\begin{proof}
  For an object $M$ of an abelian category, we denote by $0_M$ the zero morphism on $M$. Let $X \in \mathcal{A}$ be an object such that $F(X) = 0$. Then we have
  \begin{equation*}
    F(0_X) = 0_{F(X)} = \id_{F(X)} = F(\id_X).
  \end{equation*}
  Thus, if (1) holds, then $0_X = \id_X$, and therefore $X = 0$. If (2) holds, then $0_X$ is an isomorphism since $\id_{F(X)}$ is, and therefore $X = 0$. Summarizing, we have proved that either of (1) or (2) implies (3).

  Now suppose that (3) holds. If $f$ is a morphism in $\mathcal{A}$ such that $F(f) = 0$, then we have $F(\mathrm{Im}(f)) = \mathrm{Im}(F(f)) = 0$ since $F$ is exact. Thus $\mathrm{Im}(f) = 0$ and therefore $f = 0$. This implies that (1) holds.

  Similarly, if $f$ is a morphism such that $F(f)$ is an isomorphism, then we have $F(\mathrm{Ker}(f)) = \mathrm{Ker}(F(f)) = 0$ and $F(\mathrm{Coker}(f)) = \mathrm{Coker}(F(f)) = 0$ since $F$ is exact. Thus we have $\mathrm{Ker}(f) = 0$ and $\mathrm{Coker}(f) = 0$ and therefore $f$ is an isomorphism. This implies that (2) holds. The proof is done.
\end{proof}

\begin{theorem}
  \label{thm:induc-faith}
  We decompose the unit object $\unitobj \in \mathcal{C}$ as $\unitobj = \unitobj_1 \oplus \dotsb \oplus \unitobj_m$ as in \eqref{eq:unit-decomp-1}, and suppose that $\End_{\mathcal{C}}(\unitobj_i) \cong k$ for all $i = 1, \dotsc, m$. Then the following assertions are equivalent:
  \begin{enumerate}
  \item $L$ is faithful.
  \item $R$ is faithful.
  \item The full subcategory $\mathcal{C}_{i j} := \unitobj_i \otimes \mathcal{C} \otimes \unitobj_j$ is zero whenever $i \ne j$.
  \end{enumerate}
\end{theorem}
\begin{proof}
  The equivalence (1) $\Leftrightarrow$ (2) follows from $L \cong R^!$. Now let $A$ be the algebra in $\mathcal{C}^{\env}$ used to define the distinguished invertible object. By the argument in the proof of Lemma~\ref{lem:induction-L-R-1}, the faithfulness of $L$ is equivalent to the faithfulness of
  \begin{equation*}
    L': \mathcal{C} \to \mathcal{C}^{\env},
    \quad V \mapsto A \otimes (V \boxtimes \unitobj) \otimes A.
  \end{equation*}
  To show $(1) \Rightarrow (3)$, we compute, for $V \in \mathcal{C}_{i j}$ with $i \ne j$,
  \begin{align*}
    L'(V) & \cong A \otimes (\unitobj_i \boxtimes \unitobj)
    \otimes (V \boxtimes \unitobj) \otimes (\unitobj_j \boxtimes \unitobj) \otimes A \\
    & \cong A \otimes (\unitobj \boxtimes \unitobj_i)
    \otimes (V \boxtimes \unitobj) \otimes (\unitobj \boxtimes \unitobj_j) \otimes A
    = 0
  \end{align*}
  by~\eqref{eq:unit-decomp-2} and~\eqref{eq:C-env-action-2}. Thus, by Lemma~\ref{lem:faithful-exact}, $L'$ cannot be faithful if $\mathcal{C}_{i j} \ne 0$ for some $i \ne j$. In other words, (1) implies (3).

  Now we suppose that (3) holds. Then $\mathcal{C} = \mathcal{C}_{1 1} \oplus \dotsb \oplus \mathcal{C}_{m m}$. By Lemma~\ref{lem:faithful-exact}, it is sufficient to show that $L'(V) \ne 0$ for every non-zero `homogeneous' object $V \in \mathcal{C}_{i i}$ to show (1).  In a similar way as above, we compute, for $V \in \mathcal{C}_{i i}$,
  \begin{equation}
    \label{eq:induc-faith-pf-1}
    (\unitobj_i \boxtimes \unitobj_i)
    \otimes L'(V) \otimes (\unitobj_i \boxtimes \unitobj_i)
    \cong A_i \otimes (V \boxtimes \unitobj_i) \otimes A_i,
  \end{equation}
  where $A_i = (\unitobj_i \boxtimes \unitobj_i) \otimes A \otimes (\unitobj_i \boxtimes \unitobj_i)$. The object $A_i$ is non-zero, since
  \begin{equation*}
    \Hom_{\mathcal{C}^{\env}}(\unitobj \boxtimes \unitobj, A_i)
    \cong \Hom_{\mathcal{C}^{\env}}(\unitobj_i \boxtimes \unitobj_i, A)
    \cong \Hom_{\mathcal{C}}(\unitobj_i, \unitobj) \ne 0.
  \end{equation*}
  Note that every tensorand of the right-hand side of~\eqref{eq:induc-faith-pf-1} is an object of the tensor full subcategory $\mathcal{D}_i := \mathcal{C}_{i i} \boxtimes \mathcal{C}_{i i}^{\rev}$ of $\mathcal{C}^{\env}$. Since
  \begin{equation*}
    \End_{\mathcal{D}_i}(\unitobj_i \boxtimes \unitobj_i)
    \cong \End_{\mathcal{C}}(\unitobj_i) \otimes_k \End_{\mathcal{C}}(\unitobj_i) \cong k \otimes_k k \cong k,
  \end{equation*}
  the unit object of $\mathcal{D}_i$ is simple. Thus, by \eqref{eq:len-ineq}, the right-hand side of \eqref{eq:induc-faith-pf-1} is a non-zero object whenever $V \ne 0$.
\end{proof}

\begin{corollary}
  \label{cor:induc-L-R-faith}
  Both $L$ and $R$ are faithful if $\End_{\mathcal{C}}(\unitobj) \cong k$.
\end{corollary}

\subsection{The distinguished invertible object as an end}

Using the equivalence $\Phi$ given by~\eqref{eq:Lex-C-equiv-1}, we also obtain the following formula of the distinguished invertible object:

\begin{lemma}
  \label{lem:end-formula}
  $\displaystyle D \cong \int_{X \in \mathcal{C}} \Hom_{\mathcal{C}}(X, \unitobj) \cdot X$.
\end{lemma}
\begin{proof}
  Let $i_X: X \boxtimes {}^* \! X \to A$ be the universal dinatural transformation. Since the duality is an anti-equivalence, the family $\{ (i_{{}^* \! X})^* \}$ is an end. Symbolically,
  \begin{equation*}
    A^*
    \cong \int_{X \in \mathcal{C}} ({}^*X \boxtimes {}^{**} \! X)^*
    \cong \int_{X \in \mathcal{C}} X \boxtimes {}^{***} \! X.
  \end{equation*}
  By Lemma~\ref{lem:int-Hom-dual}, $A^* \cong \iHom(\unitobj, D) \cong \overline{\Phi}(D \otimes (-))$. Hence,
  \begin{equation*}
    D \cong \Phi(A^*)(\unitobj)
    \cong \int_{X \in \mathcal{C}} \Hom_{\mathcal{C}}({}^{**}X, \unitobj) \cdot X
    \cong \int_{X \in \mathcal{C}} \Hom_{\mathcal{C}}(X, \unitobj) \cdot X. \qedhere
  \end{equation*}
\end{proof}

The above lemma yields another proof of \cite[Theorem 6.1]{MR2097289}:

\begin{lemma}
  \label{lem:soc-P-unit}
  If the unit object $\unitobj \in \mathcal{C}$ is a simple object, then $D$ is isomorphic to the socle of the projective cover of $\unitobj$.
\end{lemma}
\begin{proof}
  Set $B(X, Y) = \Hom_{\mathcal{C}}(X, \unitobj) \cdot Y$. We denote by $\pi: D \dinatto B$ the universal dinatural transformation of the end. For $X_1, X_2 \in \mathcal{C}$, we have
  \begin{align*}
    \pi_{X_1 \oplus X_2}
    & = B(\id_{X_1 \oplus X_2}, i_1 p_1) \circ \pi_{X_1 \oplus X_2} + B(\id_{X_1 \oplus X_2}, i_2 p_2) \circ \pi_{X_1 \oplus X_2} \\
    & = B(p_1, i_1) \circ \pi_{X_1} + B(p_2, i_2) \circ \pi_{X_2},
  \end{align*}
  where $i_r: X_r \to X_1 \oplus X_2$ and $p_r: X_1 \oplus X_2 \to X_r$ ($r = 1, 2$) are the inclusion and the projection, respectively. Hence, if $\pi_{X_1} = 0$ and $\pi_{X_2} = 0$, then $\pi_{X_1 \oplus X_2} = 0$.

  Now let $V_0 = \unitobj, V_1, \dotsc, V_n$ be a complete set of representatives of the isomorphism classes of simple objects of $\mathcal{C}$, and let $P_i$ be the projective cover of $V_i$. Suppose that $\pi_{P_i} = 0$ for all $i$. Then, by the above argument, $\pi_P = 0$ for all projective object $P \in \mathcal{C}$. For each $X \in \mathcal{C}$, there are a projective object $P \in \mathcal{C}$ and an epimorphism $f: P \to X$. Since $B(f, X)$ is monic and $B(f, \id_X) \circ \pi_X = B(\id_P, f) \circ \pi_P = 0$, we have $\pi_X = 0$. Hence $D = 0$ by the universal property. This contradicts to the fact that $D$ is a simple object.

  By the above argument, $\pi_{P_i} \ne 0$ for some $i$. Since $B(P_i, P_i) = 0$ for $i \ne 0$, the morphism $\pi_{P_0}$ must be non-zero. Since $D$ is a simple object, the morphism $\pi_0$ induces a monomorphism from $D$ to the socle $S$ of $P_0$. On the other hand, it is known that $S$ is simple \cite{MR2119143}. Thus $D \cong S$.
\end{proof}

We consider the case where $\unitobj \in \mathcal{C}$ may not be a simple object. As in \eqref{eq:unit-decomp-1}, we decompose $\unitobj \in \mathcal{C}$ as $\unitobj = \unitobj_1 \oplus \dotsb \oplus \unitobj_m$, and define $\mathcal{C} _{i j} \subset \mathcal{C}$ as before. Note that each $\mathcal{C}_{i i}$ is a finite tensor category. Since $\mathcal{C}_{i i} \boxtimes \mathcal{C}_{i i}^{\env}$ is a tensor full subcategory of $\mathcal{C}^{\env}$, it is also a finite tensor category, and therefore the distinguished invertible object of $\mathcal{C}_{i i}$ is defined. We denote it by $D_i \in \mathcal{C}_{i i}$. The following theorem and Corollary~\ref{cor:ss-FTC-unimo} below have been conjectured in Remark 4.3.7 of \cite{2013arXiv1312.7188D}.

\begin{theorem}
  \label{thm:D-decomposition}
  $D \cong D_1 \oplus \dotsb \oplus D_m$.
\end{theorem}
\begin{proof}
  We decompose $D$ as $D = \bigoplus_{i j} a_{i j}$ ($a_{i j} \in \mathcal{C}_{i j}$). Since $D$ is invertible,
  \begin{equation*}
    \unitobj_1 \oplus \dotsb \oplus \unitobj_m
    \cong \unitobj
    \cong D \otimes D^*
    \cong \bigoplus_{i, j, p} a_{i p} \otimes a_{j p}^*.
  \end{equation*}
  From this, we see that there exists a permutation $\sigma$ on $\{ 1, \dotsc, m \}$ such that
  \begin{equation*}
    a_{i, \sigma(i)} \otimes a_{i, \sigma(i)}^* \cong \unitobj_i
    \quad \text{and} \quad
    a_{i j} = 0 \quad (j \ne \sigma(i)).
  \end{equation*}
  By the Radford $S^4$-formula \eqref{eq:Radford-S4}, we also have
  \begin{equation*}
    a_{i, \sigma(i)}
    \cong D \otimes \unitobj_{\sigma(i)}
    \cong \unitobj_{\sigma(i)}^{****} \otimes D
    \cong \unitobj_{\sigma(i)} \otimes D
    \cong a_{\sigma(i), \sigma^2(i)}
  \end{equation*}
  for $i = 1, \dotsc, m$. Thus $\sigma$ must be the trivial permutation. In conclusion,
  \begin{equation*}
    D = D'_1 \oplus \dotsb \oplus D_m'
  \end{equation*}
  for some invertible object $D'_i \in \mathcal{C}_{i i}$.

  Now we show $D_i' \cong D_i$ for $i = 1, \dotsc, m$. By Lemma~\ref{lem:end-formula}, we have
  \begin{equation*}
    D_i \cong \int_{X \in \mathcal{C}_{i i}} \Hom_{\mathcal{C}_{ii}}(X, \unitobj_i) \cdot X
    \cong \int_{X \in \mathcal{C}_{i i}} \Hom_{\mathcal{C}}(X, \unitobj) \cdot X.
  \end{equation*}
  By the universal property, there exists a morphism $\phi_i: D \to D_i$ in $\mathcal{C}$ compatible with the coend structures. If $\phi_i = 0$, then we would obtain $D_i = 0$ by a similar argument as in the proof of Lemma~\ref{lem:soc-P-unit}. Thus $\phi_i \ne 0$. Since $D_i \in \mathcal{C}_{i i}$, the morphism $\phi_i$ induces a non-zero morphism $D_i' \to D_i$. Since $D_i$ and $D_i'$ are invertible objects of $\mathcal{C}_{i i}$, they are simple. Hence, by Schur's lemma, $D_i \cong D_i'$.
\end{proof}

The following corollary is a combination of Lemma~\ref{lem:soc-P-unit} and Theorem~\ref{thm:D-decomposition}:

\begin{corollary}
  For a finite tensor category $\mathcal{C}$ with property~\eqref{eq:C-env-assume}, the distinguished invertible object of $\mathcal{C}$ is isomorphic to the socle of the projective cover of the unit object of $\mathcal{C}$.
\end{corollary}

We now have the following generalization of \cite[Corollary~6.4]{MR2097289}:

\begin{corollary}
  \label{cor:ss-FTC-unimo}
  A semisimple finite tensor category with property~\eqref{eq:C-env-assume} is unimodular.
\end{corollary}

\section{Applications, II. Constructions of topological invariants}
\label{sec:applications-2}

\subsection{General assumptions}

Throughout this section, $\mathcal{C}$ is a finite tensor category over a field $k$ satisfying \eqref{eq:C-env-assume} and
\begin{equation}
  \label{eq:simple-unit}
  \End_{\mathcal{C}}(\unitobj) \cong k.
\end{equation}
Note that this assumption implies that $\unitobj \in \mathcal{C}$ is simple. Unless otherwise noted, $D$, $U$, $L$, and $R$ have the same meaning as in \S\ref{subsec:app-1-nota}.

\subsection{A commutative algebra in the center}

The aim of this section is to give applications of our results to some constructions of topological invariants. As a preparation, we consider the algebra $B := R(\unitobj)$ obtained as the image of the trivial algebra $\unitobj$ under the monoidal functor $R$. For $V \in \mathcal{C}$, we denote by $R(V)_B \in \mathcal{Z}(\mathcal{C})_B$ the object $R(V)$ endowed with the right $B$-action given by
\begin{equation*}
  R(V) \otimes B
  = R(V) \otimes R(\unitobj)
  \xrightarrow{\quad R_2(V, \unitobj) \quad}
  R(V \otimes \unitobj) = R(V).
\end{equation*}
We note that $\mathcal{Z}(\mathcal{C})$ acts on $\mathcal{C}$ by
\begin{equation*}
  X \ogreaterthan V = U(X) \otimes V \quad (X \in \mathcal{Z}(\mathcal{C}), V \in \mathcal{C}).
\end{equation*}

\begin{theorem}
  \label{thm:com-Frob-ZC}
  The algebra $B$ has the following properties:
  \begin{enumerate}
  \item $B$ is a commutative algebra in $\mathcal{Z}(\mathcal{C})$.
  \item The following functor is an equivalence of $\mathcal{Z}(\mathcal{C})$-module categories:
    \begin{equation*}
      K: \mathcal{C} \to \mathcal{Z}(\mathcal{C})_B,
      \quad V \mapsto R(V)_B
      \quad (V \in \mathcal{C}).
    \end{equation*}
  \item $({}_B B)^* \cong R(D^*)_B$ as right $B$-modules.
  \item $B$ is a Frobenius algebra if and only if $\mathcal{C}$ is unimodular.
  \end{enumerate}
\end{theorem}
\begin{proof}
  Part (1) seems to be well-known; see, {\it e.g.}, the proof of \cite[Lemma 3.5]{MR3039775}. To prove Part (2), we note that the internal Hom functor for the $\mathcal{Z}(\mathcal{C})$-module category $\mathcal{C}$ is given by $\iHom(V, W) = R(W \otimes V^*)$ (see Example~\ref{ex:mod-cat-2}). Corollaries~\ref{cor:induction-L-R-exact} and \ref{cor:induc-L-R-faith} imply that $\iHom(\unitobj, -)$ is exact and faithful. Thus, by Theorem~\ref{thm:mod-cat-comparison}, the functor $K$ is an equivalence of $\mathcal{Z}(\mathcal{C})$-module categories.

  By Corollary~\ref{cor:induction-L-R-2}, $U \dashv R \dashv U' := D^* \otimes U(-)$. Let $F_B: \mathcal{Z}(\mathcal{C})_B \to \mathcal{Z}(\mathcal{C})$ be the forgetful functor. Obviously, $F_B \circ K = R$. Since $K$ is an equivalence, we have
  \begin{equation}
    \label{eq:com-Frob-ZC-pf-1}
    K \circ U \ \dashv \ F_B \ \dashv \ K \circ U'.
  \end{equation}
  The functor $(-) \otimes (B_B)^*$ is also right adjoint to $F_B$ by Lemma \ref{lem:adj-restriction}. Hence there exists a natural isomorphism
  \begin{equation*}
    K(D^* \otimes U(X)) \cong X \otimes ({}_B B)^*
    \quad (X \in \mathcal{Z}(\mathcal{C}))
  \end{equation*}
  of $B$-modules. Part (3) is proved by letting $X = \unitobj$ in the above formula. Part (4) follows from the following logical equivalences:
  \begin{equation*}
    \text{$B$ is Frobenius}
    \iff \text{$F_B$ is Frobenius}
    \overset{\text{\eqref{eq:com-Frob-ZC-pf-1}}}{\iff} U \cong U'
    \iff D \cong \unitobj. \qedhere
  \end{equation*}
\end{proof}

\begin{remark}
  \label{rem:Fb-alg-ZC-tr}
  Suppose that $\mathcal{C}$ is unimodular. By the above theorem, there exists a morphism $\lambda_0: B \to \unitobj$ such that $(B, \lambda_0)$ is a Frobenius algebra. It is easy to see that $\lambda_0 \ne 0$ and $(B, c \lambda_0)$ is Frobenius for any $c \in k^{\times}$. Since
  \begin{equation*}
    \dim_k \Hom_{\mathcal{Z}(\mathcal{C})}(B, \unitobj)
    = \dim_k \Hom_{\mathcal{C}}(D, \unitobj)
    = \dim_k \End_{\mathcal{C}}(\unitobj) = 1
  \end{equation*}
  by Corollary~\ref{cor:induction-L-R-2} and our assumption~\eqref{eq:simple-unit}, we have the following conclusion: The pair $(B, \lambda)$ is a commutative Frobenius algebra in $\mathcal{Z}(\mathcal{C})$ for {\em any} non-zero morphism $\lambda: B \to \unitobj$ in $\mathcal{Z}(\mathcal{C})$.
\end{remark}

\subsection{Invariants of handlebody-links}
\label{subsec:hb-link-inv}

Let $\mathbb{N}_0 = \{ 0, 1, 2, \dotsc \}$ denote the set of non-negative integers. A {\em handlebody} of genus $g \in \mathbb{N}_0$ is a 3-manifold obtained from a $3$-ball by attaching $g$ handles. For each $n \in \mathbb{N}_0$, we fix a subset $D_n \subset \mathbb{R}^2$ consisting of $n$ disjoint unit disks whose center lies on the $x$-axis. For $n, m \in \mathbb{N}$, an {\em $(n, m)$-handlebody-tangle} is a disjoint union $T = T_1 \sqcup \dotsb \sqcup T_r$ of handlebodies embedded into $\mathbb{R}^2 \times [0, 1]$ such that
\begin{equation*}
  T \cap (\mathbb{R}^2 \times \{1\}) = D_n, \quad
  T \cap (\mathbb{R}^2 \times \{0\}) = D_m,
\end{equation*}
and the intersection of every genus zero component of $T$ and $\mathbb{R}^2 \times \{ 0, 1 \}$ consists of more than two disks. A {\em handlebody-link} is a $(0, 0)$-handlebody-tangle. By convention, we regard the empty set as a handlebody-link.

Ishii and Masuoka \cite{MR3265394} introduced the braided monoidal category $\mathcal{T}$ of handlebody-tangles. By definition, the class of objects of $\mathcal{T}$ is the set $\mathbb{N}_0$, and the set of morphisms from $n$ to $m$ in $\mathcal{T}$ is the set of equivalence classes of $(n, m)$-handlebody-tangles (here, two handlebody-tangles are said to be {\em equivalent} if one can be transformed into the other by a boundary-preserving isotopy $\mathbb{R}^2 \times [0, 1]$). The composition of morphisms, the tensor product and the braiding of $\mathcal{T}$ are defined in a similar way as the category of ordinary (framed) tangles.

Now let $\mathcal{B}$ be a braided monoidal category with braiding $\sigma$. By the definition of the category $\mathcal{T}$, a braided monoidal functor $\mathcal{T} \to \mathcal{B}$ yields an invariant of handlebody-links with values in $\End_{\mathcal{B}}(\unitobj)$. To construct such a functor, Ishii and Masuoka \cite{MR3265394} introduced the following notion:

\begin{definition}
  \label{def:QCQSA}
  A {\em quantum-commutative quantum-symmetric algebra} (QCQSA) in $\mathcal{B}$ is a triple $(A, m, e)$ consisting of an object $A \in \mathcal{B}$ and morphisms
  \begin{equation*}
    m: A \otimes A \to A \quad \text{and} \quad
    e: A \otimes A \to \unitobj
  \end{equation*}
  satisfying the following conditions:
  \begin{enumerate}
    \renewcommand{\labelenumi}{(Q\arabic{enumi})}
  \item $m \circ (m \otimes \id_A) = m \circ (\id_A \otimes m)$
  \item $e \circ (m \otimes \id_A) = e \circ (\id_A \otimes m)$.
  \item $m$ is {\em commutative}, {\it i.e.}, $m \circ \sigma_{A,A} = m$.
  \item $e$ is {\em symmetric}, {\it i.e.}, $e \circ \sigma_{A,A} = e$.
  \item There exists a morphism $c: \unitobj \to A \otimes A$ such that the triple $(A, e, c)$ is a left dual object of $A$.
  \end{enumerate}
\end{definition}

\begin{figure}
  \renewcommand{\tabcolsep}{1em}
  \begin{tabular}{ccccc}
    \includegraphics{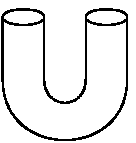} &
    \includegraphics{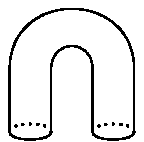} &
    \includegraphics{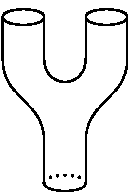} &
    \includegraphics{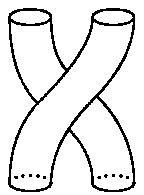} &
    \includegraphics{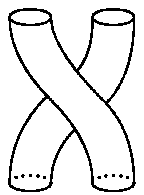} \\ 
    $\cup: 2 \to 0$ &
    $\cap: 0 \to 2$ &
    $\mathsf{Y}: 2 \to 1$ &
    $\mathsf{X}^+: 2 \to 2$ &
    $\mathsf{X}^-: 2 \to 2$
  \end{tabular}
  \caption{Elementary handlebody-tangles}
  \label{fig:elem-hb-tang}
\end{figure}

It is easy to see that the monoidal category $\mathcal{T}$ is generated by handlebody-tangles $\cup$, $\cap$, $\mathsf{Y}$, $\mathsf{X}^+$ and $\mathsf{X}^-$ depicted in Figure~\ref{fig:elem-hb-tang}. The fundamental relations among these generators are completely determined in \cite{MR2943680,MR3265394}. As a result, a QCQSA $(A, m, e)$ in $\mathcal{B}$ yields a unique (up to isomorphism) braided monoidal functor $F: \mathcal{B} \to \mathcal{T}$ such that $F(1) = A$, $F(\cup) = e$, $F(\cap) = c$, $F(\mathsf{Y}) = m$ and $F(\mathsf{X}^+) = \sigma_{A,A}$, where $c$ is the morphism in (Q5) of Definition~\ref{def:QCQSA}.

If $(A, \lambda)$ is a commutative Frobenius algebra in $\mathcal{B}$, then $(A, m_A, \lambda \circ m_A)$ is a QCQSA in $\mathcal{B}$ (moreover, every `unital' QCQSA is obtained in this way). Now we suppose that $\mathcal{C}$ is unimodular. Then, by Theorem~\ref{thm:com-Frob-ZC}, the algebra $B = R(\unitobj)$ is commutative and Frobenius. Given a trace $\lambda: B \to \unitobj$, we denote by
\begin{equation*}
  F_{\mathcal{C}}(-; \lambda): \mathcal{T} \to \mathcal{Z}(\mathcal{C})
\end{equation*}
the braided monoidal functor obtained from the QCQSA $(B, m_B, \lambda \circ m_B)$ by the above construction.

\begin{example}
  \label{ex:com-Fb-alg-YD}
  Let $H$ be a finite-dimensional Hopf algebra over $k$ with comultiplication $\Delta$, counit $\epsilon$ and antipode $S$. We use the Sweedler notation, such as
  \begin{equation*}
    \Delta(h) = h_{(1)} \otimes h_{(2)}
    \quad \text{and} \quad
    \Delta(h_{(1)}) \otimes h_{(2)} = h_{(1)} \otimes h_{(2)} \otimes h_{(3)}
    = h_{(1)} \otimes \Delta(h_{(2)})
  \end{equation*}
  for $h \in H$. A {\em Yetter-Drinfeld module} over $H$ is a left $H$-module $M$ endowed with a left $H$-comodule structure $m \mapsto m_{(-1)} \otimes m_{(0)}$ such that
  \begin{equation*}
    (h m)_{(-1)} \otimes (h m)_{(0)}
    = h_{(1)} m_{(-1)} S(h_{(3)}) \otimes h_{(2)} m
  \end{equation*}
  for all $h \in H$ and $m \in M$ \cite{MR1243637}. As is well-known, the center of $\mathcal{C} := H\mbox{-{\sf mod}}$ can be identified with the category ${}^H_H \mathcal{YD}_f$ of finite-dimensional Yetter-Drinfeld modules over $H$. Under this identification, a right adjoint $R: \mathcal{C} \to {}^H_H \mathcal{YD}$ of $U$ is given as follows: As a vector space, $R(V) = H \otimes_k V$ for $V \in \mathcal{C}$. The action and the coaction of $H$ on $R(V)$ are given by
  \begin{equation*}
    h \cdot (a \otimes v) = h_{(1)} a S(h_{(3)}) \otimes (h_{(2)} \cdot v)
    \quad \text{and} \quad
    a \otimes v \mapsto a_{(1)} \otimes a_{(2)} \otimes v,
  \end{equation*}
  respectively, for $h, a \in H$ and $v \in V$. We note that the unit $\eta^r: \id_{\mathcal{Z}(\mathcal{C})} \to R U$ and the counit $\varepsilon^r: U R \to \id_{\mathcal{C}}$ of the adjunction are given by
  \begin{equation*}
    \eta^r_M(m) = m_{(-1)} \otimes m_{(0)}
    \quad \text{and} \quad
    \varepsilon^r_V(a \otimes v) = \epsilon(a) v,
  \end{equation*}
  respectively, for $m \in M \in {}^H_H \mathcal{YD}_f$, $v \in V \in \mathcal{C}$ and $a \in H$. The algebra $B = R(k)$ is identical to the one considered in \cite{MR3265394}.
\end{example}

Let $H$ and $B$ be as in Example~\ref{ex:com-Fb-alg-YD}. A linear map $\lambda: B \to k$ is $H$-colinear if and only if $h_{(1)} \lambda(h_{(2)}) = \lambda(h) 1$ for all $h \in H$, {\it i.e.}, $\lambda$ is a left integral on $H$. Thus, in view of Remark~\ref{rem:Fb-alg-ZC-tr}, a non-zero left integral $\lambda$ on $H$ is a morphism $\lambda: B \to k$ in ${}^H_H \mathcal{YD}$ if and only if $H$ is unimodular.

Now we suppose that $H$ is unimodular. Let $\lambda: H \to k$ be a non-zero left integral on $H$. By the above argument, we obtain a braided monoidal functor
\begin{equation*}
  F_H(-; \lambda) := F_{H\text{-{\sf mod}}}(-; \lambda)
\end{equation*}
from $\mathcal{T}$ to $\mathcal{Z}(H\mbox{-{\sf mod}}) \cong {}^H_H \mathcal{YD}_f$. Restricting this functor to $\End_{\mathcal{T}}(0)$, we obtain an invariant of handlebody-links. However, as Ishii and Masuoka observed in \cite{MR3265394}, the invariant obtained in this way is constantly zero unless $H$ is cosemisimple.

To obtain a meaningful invariant from non-cosemisimple $H$, they proposed the following modification of the above invariant: Every handlebody-link $T$ can be expressed as $T = \cup \circ T'$ for some $T': 0 \to 2$ in $\mathcal{T}$. Choose such $T'$ and set
\begin{equation}
  \label{eq:inv-V-Hopf}
  V_H(T; \lambda) := \epsilon \circ F_H(T'; \lambda),
\end{equation}
where $\epsilon: H \to k$ is the counit of $H$. If the condition
\begin{equation}
  \label{eq:S-inv-Hopf}
  \lambda(S(z)) = \lambda(z) \quad \text{for all $z \in \mathrm{Cent}(H)$} \quad \text{($:=$ the center of $H$)}
\end{equation}
is satisfied, then $V_H(T; \lambda)$ does not depend on the choice of $T'$ and hence $V_H(-; \lambda)$ is an invariant of handlebody-links.

Although it is not yet known whether $V_H(-; \lambda)$ is a non-trivial invariant, it is an interesting problem to understand their construction in the setting of finite tensor categories. To attack this problem, we utilize the central Hopf monad $Z$ on $\mathcal{C}$. Let $\mu$, $\eta$, $Z_2$, $Z_0$ and $S$ have the same meaning as \S\ref{subsec:central-Hopf-monad}. To formulate the condition \eqref{eq:S-inv-Hopf} in the categorical setting, we consider the following map:
\begin{equation*}
  \mathfrak{S}: \Hom_{\mathcal{C}}(Z(\unitobj), \unitobj)
  \to \Hom_{\mathcal{C}}(Z(\unitobj), \unitobj),
  \quad \alpha \mapsto S_{\unitobj} \circ Z(\alpha^*).
\end{equation*}
Note that there is an isomorphism
\begin{equation}
  \label{eq:Z-and-End-id}
  \Hom_{\mathcal{C}}(Z(\unitobj), \unitobj)
  \to \END(\id_{\mathcal{C}}) := \NAT(\id_{\mathcal{C}}, \id_{\mathcal{C}}),
  \quad \alpha \mapsto (\id_{(-)} \otimes \alpha) \circ \partial_{\unitobj}(-),
\end{equation}
where $\partial_V(X): V \otimes X \to X \otimes Z(V)$ is a natural transformation given in \S\ref{subsec:Hopf-monads}. The map $\mathfrak{S}$ has the following meaning:

\begin{lemma}
  \label{lem:S-operator}
  For a natural transformation $\xi: \id_{\mathcal{C}} \to \id_{\mathcal{C}}$, we define
  \begin{equation*}
    {}^!\xi_X = (\xi_{{}^* \! X})^* \quad (V \in \mathcal{C}).
  \end{equation*}
  Then the following diagram commutes:
  \begin{equation*}
    \xymatrix{
      \Hom_{\mathcal{C}}(Z(\unitobj), \unitobj)
      \ar[rr]^{\eqref{eq:Z-and-End-id}}
      \ar[d]_{\mathfrak{S}}
      & & \END(\id_{\mathcal{C}})\phantom{.}
      \ar[d]^{{}^!(-)} \\
      \Hom_{\mathcal{C}}(Z(\unitobj), \unitobj)
      \ar[rr]_{\eqref{eq:Z-and-End-id}} & & \END(\id_{\mathcal{C}}).
    }
  \end{equation*}
\end{lemma}
\begin{proof}
  For all $\alpha \in \Hom_{\mathcal{C}}(Z(\unitobj), \unitobj)$ and $X \in \mathcal{C}$, we have
  \begin{align*}
    (\id_X \otimes \mathfrak{S}(\alpha)) \circ \partial_{\unitobj}(X)
    & = (\id_X \otimes S_{\unitobj}) \circ (\id_X \otimes Z(\alpha^*)) \circ \partial_{\unitobj}(X) \\
    & = (\id_X \otimes S_{\unitobj}) \circ \partial_{Z(\unitobj)^*}(X) \circ (\alpha^* \otimes \id_X) \\
    & = \partial_{\unitobj}({}^* \! X)^* \circ (\id_{{}^* \! X} \otimes \alpha)^*
    \qquad \text{(by \eqref{eq:Hopf-monad-Z-antipode-2})} \\
    & = ((\id_{{}^* \! X} \otimes \alpha) \circ \partial_{\unitobj}({}^*X))^*. \qedhere
  \end{align*}
\end{proof}

Under the identification $\mathcal{Z}(\mathcal{C}) \approx {}_Z \mathcal{C}$, the functor
\begin{equation*}
  L: \mathcal{C} \to \mathcal{Z}(\mathcal{C}),
  \quad V \mapsto (Z(V), \mu_V)
\end{equation*}
is left adjoint to $U$. In view of the results of \S\ref{subsec:colax-lax-adj}, we may assume $R = L^!$. Let $B$ be the commutative Frobenius algebra of Theorem~\ref{thm:com-Frob-ZC} with trace $\lambda$. As a categorical counter-part of~\eqref{eq:S-inv-Hopf}, we introduce the following condition:
\begin{equation}
  \label{eq:S-inv-FTC}
  \mathfrak{S}(\alpha) \circ \lambda^* = \alpha \circ \lambda^*
  \quad \text{for all $\alpha \in \Hom_{\mathcal{C}}(Z(\unitobj), \unitobj)$}.
\end{equation}

\begin{theorem}
  \label{thm:IM-inv}
  Notations are as above. Given a handlebody-link $T$, we choose a handlebody-tangle $T'$ such that $T = \cup \circ T'$ and then set
  \begin{equation}
    \label{eq:inv-V-FTC}
    V_{\mathcal{C}}(T; \lambda) = \varepsilon^r_{\unitobj} \circ F_{\mathcal{C}}(T'; \lambda),
  \end{equation}
  where $\varepsilon^r: U R \to \id_{\mathcal{C}}$ is the counit of the adjunction $U \dashv R$. If \eqref{eq:S-inv-FTC} is satisfied, then $V_{\mathcal{C}}(T; \lambda)$ does not depend on the choice of $t'$ and hence $V_{\mathcal{C}}(-; \lambda)$ is an invariant of handlebody-links.
\end{theorem}
\begin{proof}
  By the same argument as in \cite[\S5]{MR3265394}, the claim reduces to that
  \begin{equation}
    \label{eq:IM-inv-pf-1}
    (\varepsilon^r_{\unitobj} \otimes \lambda) \circ \beta
    = (\lambda \otimes \varepsilon^r_{\unitobj}) \circ \beta
    \quad \text{for all $\beta \in \Hom_{\mathcal{Z}(\mathcal{C})}(\unitobj, B \otimes B)$}.
  \end{equation}
  Let $\eta^{\ell}$ and $\varepsilon^{\ell}$ be the unit and the counit of $L \dashv U$, respectively. We may assume that the unit $\eta^r$ and the counit $\varepsilon^r$ of $U \dashv R$ are given by
  \begin{equation*}
    \eta^r_M = {}^*(\varepsilon^{\ell}_{M^*})
    \quad (M \in {}_Z \mathcal{C})
    \quad \text{and} \quad
    \varepsilon^r_V = {}^*(\eta^{\ell}_{V^*})
    \quad (V \in \mathcal{C}),
  \end{equation*}
  respectively. Then~\eqref{eq:IM-inv-pf-1} is equivalent to that
  \begin{equation}
    \label{eq:IM-inv-pf-2}
    \gamma \circ (\lambda^* \otimes \eta_{\unitobj}^{\ell})
    = \gamma \circ (\eta_{\unitobj}^{\ell} \otimes \lambda^*)
    \quad \text{for all $\gamma \in \Hom_{Z}(Z(\unitobj) \otimes Z(\unitobj), \unitobj)$}.
  \end{equation}
  We note that there is a canonical isomorphism
  \begin{equation*}
    \Theta: \Hom_{\mathcal{C}}(Z(\unitobj), \unitobj)
    \xrightarrow{\quad (-)^* \quad}
    \Hom_{\mathcal{C}}(\unitobj, Z(\unitobj)^*)
    \xrightarrow{\quad L \, \dashv \, U \quad}
    \Hom_{Z}(Z(\unitobj), Z(\unitobj)^*)
  \end{equation*}
  Explicitly, $\Theta(\alpha) = S_{Z(\unitobj)} \circ Z(\mu_{\unitobj}^*) \circ Z(\alpha^*)$ by \eqref{eq:Hopf-monad-dual-module}. Now let $\gamma: Z(\unitobj) \otimes Z(\unitobj) \to \unitobj$ be a morphism of $Z$-modules. Since
  \begin{equation*}
    \gamma = \eval_{Z(\unitobj)} \circ (\Theta(\alpha) \otimes \id_{Z(\unitobj)})
  \end{equation*}
  for some $\alpha: Z(\unitobj) \to \unitobj$, we compute
  \begin{align*}
    \gamma \circ (\eta_{\unitobj}^{\ell} \otimes \lambda^*)
    & = \eval_{Z(\unitobj)} \circ (\Theta(\alpha) \otimes \id_{Z(\unitobj)})
    \circ (\eta_{\unitobj}^{\ell} \otimes \lambda^*) \\
    & = \eta_{\unitobj}^* \circ S_{Z(\unitobj)}
    \circ Z(\mu_{\unitobj}^*) \circ Z(\alpha^*) \circ \lambda^* \\
    & = S_{\unitobj} \circ Z(Z(\eta_{\unitobj})^*)
    \circ Z(\mu_{\unitobj}^*) \circ Z(\alpha^*) \circ \lambda^*
    & & \text{(by the naturality of $S$)} \\
    & = S_{\unitobj} \circ Z(\alpha^*) \circ \lambda^*
    & & \text{(by $\mu_{\unitobj} \circ Z(\eta_{\unitobj}) = \id_{Z(\unitobj)}$)} \\
    & = \mathfrak{S}(\alpha) \circ \lambda^*, \\
    \gamma \circ (\lambda^* \otimes \eta_{\unitobj}^{\ell})
    & = \lambda^{**} \circ S_{Z(\unitobj)}
    \circ Z(\mu_{\unitobj}^*) \circ Z(\alpha^*) \circ \eta_{\unitobj} \\
    & = S_{\unitobj} \circ Z(Z(\lambda^*)^*)
    \circ Z(\mu_{\unitobj}^*) \circ Z(\alpha^*) \circ \eta_{\unitobj}
    & & \text{(by the naturality of $S$)} \\
    & = S_{\unitobj} \circ Z((\mu_{\unitobj} \circ Z(\lambda^*))^*)
    \circ Z(\alpha^*) \circ \eta_{\unitobj} \\
    & = S_{\unitobj} \circ Z((\lambda^* \circ Z_0)^*)
    \circ Z(\alpha^*) \circ \eta_{\unitobj}
    & & \text{(by the $Z$-linearity of $\lambda^*$)} \\
    & = S_{\unitobj} \circ Z(Z_0^*) \circ \eta_{Z(\unitobj)^*}
    \circ \lambda^{**} \circ \alpha^*
    & & \text{(by the naturality of $\eta$)} \\
    & = \lambda^{**} \circ \alpha^*
    & & \text{(see \cite[\S2.3 and \S3.3]{MR2355605})} \\
    & = \alpha \circ \lambda^{*}
    & & \text{($f^* = f$ for $f \in \End_{\mathcal{C}}(\unitobj)$)}.
  \end{align*}
  This means that~\eqref{eq:IM-inv-pf-2} is equivalent to \eqref{eq:S-inv-FTC}.
\end{proof}

\begin{example}
  We use the same notation as in Example~\ref{ex:com-Fb-alg-YD}. Suppose that $H$ is unimodular. Then, as we have seen, a non-zero right integral $\lambda$ on $H$ is a trace of the algebra $B$ in ${}^H_H \mathcal{YD}_f$. It is easy to see that \eqref{eq:inv-V-FTC} reduces to \eqref{eq:inv-V-Hopf} if $\mathcal{C} = H\mbox{-{\sf mod}}$. To see that \eqref{eq:S-inv-FTC} reduces to \eqref{eq:S-inv-Hopf}, we note that there are isomorphisms
  \begin{equation*}
    \Hom_{\mathcal{C}}(Z(\unitobj), \unitobj) \cong \END(\id_{\mathcal{C}}) \cong \mathrm{Cent}(H).
  \end{equation*}
  If $z \in \mathrm{Cent}(H)$ corresponds to $f \in \Hom_{\mathcal{C}}(Z(\unitobj), \unitobj)$, then, by Lemma~\ref{lem:S-operator},
  \begin{equation*}
    \lambda^* \circ f = \lambda(z)
    \quad \text{and} \quad
    \lambda^* \circ \mathfrak{S}(f) = \lambda(S^{-1}(z))
  \end{equation*}
  in $\End_H(k) \cong k$. Therefore \eqref{eq:S-inv-FTC} is equivalent to \eqref{eq:S-inv-Hopf} if $\mathcal{C} = H\mbox{-{\sf mod}}$.
\end{example}

\subsection{Integral of the coend Hopf algebra}

Suppose that the finite tensor category $\mathcal{C}$ has a braiding $\sigma$. Then the object
\begin{equation*}
  \mathbf{F}_{\mathcal{C}} := Z(\unitobj) = \int^{X \in \mathcal{C}} X^* \otimes X
\end{equation*}
has a structure of a Hopf algebra in $\mathcal{C}$ given as follows: The comultiplication $\Delta$ and the counit $\varepsilon$ are induced from the comonoidal structure of the Hopf monad $Z$. The multiplication and the unit are defined by
\begin{equation*}
  m \circ (i_\unitobj(X) \otimes i_\unitobj(Y))
  = i_\unitobj(Y \otimes X) \circ (\id_{X^*} \otimes \sigma_{X,Y^* \otimes Y})
  \text{\quad and \quad}
  u = i_\unitobj(\unitobj),
\end{equation*}
respectively, for $X, Y \in \mathcal{C}$, where $i_V(X): X^* \otimes V \otimes X \to Z(V)$ ($V, X \in \mathcal{C}$) is the universal dinatural transformation. We omit the description of the antipode since we will not use it. The Hopf algebra $\mathbf{F}_{\mathcal{C}}$ is used to construct an invariant of closed 3-manifolds; see \cite{MR1862634} and \cite{MR2251160}.

In this subsection, we give some applications of our results to integrals for $\mathbf{F}_{\mathcal{C}}$. We first recall the definition of an integral in a braided Hopf algebra. Let $\mathcal{B}$ be a braided monoidal category, and let $H$ be a Hopf algebra in $\mathcal{B}$. A ($K$-based) {\em left integral in $H$} \cite[Definition 3.1]{MR1759389} is a pair $(K, \Lambda)$ consisting of an object $K \in \mathcal{B}$ and a morphism $\Lambda: K \to H$ in $\mathcal{B}$ satisfying
\begin{equation}
  \label{eq:left-integral}
  m_H \circ (\id_H \otimes \Lambda) = \varepsilon_H \otimes \Lambda,
\end{equation}
where $m_H$ and $\varepsilon_H$ are the multiplication and the counit of $H$, respectively. We denote by $\mathcal{I}_{\ell}(H)$ the full subcategory of the category of objects over $H$ and call it the category of left integrals in $H$. The category $\mathcal{I}_r(H)$ of right integrals in $H$ is defined in a similar way.

Now suppose that $\mathcal{B}$ is rigid and has equalizers. Then the antipode $S$ of $H$ is invertible \cite[Theorem 4.1]{MR1685417}. The category $\mathcal{I}_{\ell}(H)$ has a terminal object \cite[Proposition 3.1]{MR1759389}. We write it as $(\mathrm{Int}(H), \Lambda_{\ell})$ and call $\mathrm{Int}(H)$ the {\em object of integrals}. The first result of this subsection is the following description of the object of integrals of the Hopf algebra $\mathbf{F}_{\mathcal{C}}$.

\begin{theorem}
  \label{thm:int-F-D}
  $\mathrm{Int}(\mathbf{F}_{\mathcal{C}}) \cong D^*$
\end{theorem}
\begin{proof}
  If we identify $\mathcal{Z}(\mathcal{C})$ with the category of $Z$-modules, then
  \begin{equation}
    L: \mathcal{C} \to \mathcal{Z}(\mathcal{C}),
    \quad V \mapsto (Z(V), \mu_V)
  \end{equation}
  is a left adjoint of $U$. Thus, $\mathbf{F}_{\mathcal{C}} \cong U L(\unitobj)$ as coalgebras. By Lemma~\ref{lem:colax-lax-adj-2},
  \begin{equation*}
    A := U R(\unitobj) \cong U({}^*L(\unitobj)) \cong {}^* U L(\unitobj) \cong {}^* \mathbf{F}_{\mathcal{C}}
  \end{equation*}
  as algebras. We note that ${}^* \mathbf{F}_{\mathcal{C}}$ is also a Hopf algebra in $\mathcal{C}$. Using basic properties of integrals for braided Hopf algebras proved in \cite{MR1759389,MR1862634,MR1685417}, we have the following logical equivalences:
  \begin{align*}
    \mathrm{Int}(\mathbf{F}_{\mathcal{C}}) \cong D^*
    \iff \mathrm{Int}({}^* \mathbf{F}_{\mathcal{C}}) \cong D
    \iff \text{$A$ has a $D$-valued trace}.
  \end{align*}
  Now set $B = R(\unitobj)$ (so that $A = U(B)$) and let $K: \mathcal{C} \to \mathcal{Z}(C)_B$ be the equivalence of $\mathcal{Z}(\mathcal{C})$-module categories given in Theorem~\ref{thm:com-Frob-ZC}. Since $\mathcal{C}$ is braided, there exists an object $\widetilde{D} \in \mathcal{Z}(\mathcal{C})$ such that $U(\widetilde{D}) = D^*$. By Theorem~\ref{thm:com-Frob-ZC}, we have
  \begin{equation*}
    ({}_B B)^*
    \cong K(D^*)
    = K U (\widetilde{D})
    = K(\widetilde{D} \ogreaterthan \unitobj)
    \cong \widetilde{D} \otimes B_B
  \end{equation*}
  as right $B$-modules. Applying $U$, we get an isomorphism $({}_A A)^* \cong D^* \otimes A_A$ of right $A$-modules. Thus $A$ has a $D$-valued trace.
\end{proof}

Suppose that $\mathcal{C}$ is unimodular. Then, by Remark~\ref{rem:Fb-alg-ZC-tr}, there exists a unique (up to scalar multiple) non-zero morphism $\Lambda: \unitobj \to Z(\unitobj)$ of $Z$-modules. The following theorem means that the pair $(\unitobj, \Lambda)$ is a `two-sided' and `universal' integral.

\begin{theorem}
  \label{thm:unimo-FTC-alg-K-elem}
  $(\unitobj, \Lambda)$ is a terminal object of both $\mathcal{I}_{\ell}(\mathbf{F}_{\mathcal{C}})$ and $\mathcal{I}_r(\mathbf{F}_{\mathcal{C}})$.
\end{theorem}
\begin{proof}
  We first show that $(\unitobj, \Lambda)$ is a left integral in $H$. By~\eqref{eq:Hopf-monad-Z-mult}, we have
  \begin{align*}
    m \circ (i_{\unitobj}(X) \otimes i_{\unitobj}(Y))
    & = i_{\unitobj}(Y \otimes X) \circ (\id_{X^*} \otimes \sigma_{X, Y^* \otimes Y}) \\
    & = \mu_{\unitobj} \circ i_{Z(\unitobj)}(X) \circ (\id_{X^*} \otimes i_{\unitobj}(Y) \otimes \id_X) \circ (\id_{X^*} \otimes \sigma_{X, Y^* \otimes Y}) \\
    & = \mu_{\unitobj} \circ i_{Z(\unitobj)}(X) \circ (\id_{X^*} \otimes \sigma_{X, Z(\unitobj)}) \circ (\id_{X^*} \otimes \id_X \otimes i_{\unitobj}(Y))
  \end{align*}
  for all $X, Y \in \mathcal{C}$. By the Fubini theorem for coends, $\{ \id_{X^*} \otimes \id_X \otimes i_{\unitobj}(Y) \}_{Y \in \mathcal{C}}$ is a coend of the functor $(Y_1, Y_2) \mapsto X^* \otimes X \otimes Y_1^* \otimes Y_2$. Thus,
  \begin{equation}
    \label{eq:unimo-FTC-alg-K-elem-1}
    m \circ (i_{\unitobj}(X) \otimes \id_{Z(\unitobj)})
    = \mu_{\unitobj} \circ i_{Z(\unitobj)}(X) \circ (\id_{X^*} \otimes \sigma_{X, Z(\unitobj)})
  \end{equation}
  for all $X \in \mathcal{C}$. By using this formula, we compute
  \begin{align*}
    m \circ (\id_F \otimes \Lambda) \circ i_{\unitobj}(X)
    & = \mu_{\unitobj} \circ i_{Z(\unitobj)}(X) \circ (\id_{X^*} \otimes \sigma_{X,Z(\unitobj)})
    \circ (\id_{X^*} \otimes \id_X \otimes \Lambda) \\
    & = \mu_{\unitobj} \circ i_{Z(\unitobj)}(X) \circ (\id_{X^*} \otimes \Lambda \otimes \id_X)
    \circ (\id_{X^*} \otimes \sigma_{X,\unitobj}) \\
    & = \mu_{\unitobj} \circ Z(\Lambda) \circ i_{\unitobj}(X) \\
    & = \Lambda \circ Z_0 \circ i_{\unitobj}(X).
  \end{align*}
  Here, the last equality follows from the assumption that $\Lambda: \unitobj \to Z(\unitobj)$ is a morphism of $Z$-modules. Recall that the counit of $\mathbf{F}_{\mathcal{C}}$ is given by $\varepsilon = Z_0$. Hence,
  \begin{equation*}
    m \circ (\id_{\mathbf{F}_{\mathcal{C}}} \otimes \Lambda)
    = \Lambda \circ Z_0
    = \Lambda \circ \varepsilon
    = (\Lambda \otimes \id_{\unitobj}) \circ (\id_{\unitobj} \otimes \varepsilon)
    = \Lambda \otimes \varepsilon,
  \end{equation*}
  {\it i.e.}, the pair $(\unitobj, \Lambda)$ is a left integral in $\mathbf{F}_{\mathcal{C}}$. To show that it is also a right integral in $H$, we remark the following description of the multiplication:
  \begin{align*}
    m \circ (i_{\unitobj}(X) \otimes i_{\unitobj}(Y))
    & = i_{\unitobj}(Y \otimes X) \circ (\id_{X^*} \otimes \sigma_{X, Y^* \otimes Y}) \\
    & = i_{\unitobj}(Y \otimes X) \circ (\id_{X^* \otimes Y^*} \otimes \sigma_{X, Y})
    \circ (\id_{X^*} \otimes \sigma_{X, Y^*} \otimes \id_{Y}) \\
    & = i_{\unitobj}(X \otimes Y) \circ (\sigma_{X^*,Y^*} \otimes \id_{X \otimes Y})
    \circ (\id_{X^*} \otimes \sigma_{X, Y^*} \otimes \id_{Y}) \\
    & = i_{\unitobj}(X \otimes Y) \circ (\sigma_{X^* \otimes X, Y^*} \otimes \id_Y).
  \end{align*}
  Here, the third equality follows from the dinaturality of $i_{\unitobj}$ and the well-known formula $\sigma_{X,Y}^* = \sigma_{X^*,Y^*}$. In a similar way as~\eqref{eq:unimo-FTC-alg-K-elem-1}, we obtain
  \begin{equation}
    \label{eq:unimo-FTC-alg-K-elem-2}
    m \circ (\id_{Z(\unitobj)} \otimes i_{\unitobj}(Y))
    = \mu_{\unitobj} \circ i_{Z(\unitobj)}(Y) \circ (\sigma_{Z(\unitobj), Y^*} \otimes \id_Y)
  \end{equation}
  for all $Y \in \mathcal{C}$. One can show that $(\unitobj, \Lambda)$ is a right integral in a similar way as above but by using \eqref{eq:unimo-FTC-alg-K-elem-2} instead of~\eqref{eq:unimo-FTC-alg-K-elem-1}.

  Now we prove that $(\unitobj, \Lambda)$ is terminal both in $\mathcal{I}_{\ell}(\mathbf{F}_{\mathcal{C}})$. By definition, there exists a morphism $f: \unitobj \to \mathrm{Int}(\mathbf{F}_{\mathcal{C}})$ such that $\Lambda_{\ell} \circ f = \Lambda$. Obviously, $f \ne 0$. By Schur's lemma and our assumption~\eqref{eq:simple-unit}, $f$ is an isomorphism. Thus,
  \begin{equation*}
    (\unitobj, \Lambda) \cong (\mathrm{Int}(\mathbf{F}_{\mathcal{C}}), \Lambda_{\ell})
  \end{equation*}
  in the category of objects over $H$, {\it i.e.}, $(\unitobj, \Lambda)$ is a terminal object in $\mathcal{I}_{\ell}(\mathbf{F}_{\mathcal{C}})$. One can show that $(\unitobj, \Lambda)$ is a terminal object of $\mathcal{I}_r(\mathbf{F}_{\mathcal{C}})$ in a similar way.
\end{proof}

\begin{remark}
  \label{rem:HKR-inv}
  Suppose that $\mathcal{C}$ is a unimodular ribbon finite tensor category. By \cite[Proposition 4.10]{MR1759389}, the morphism $\Lambda: \unitobj \to \mathbf{F}_{\mathcal{C}}$ in Theorem~\ref{thm:unimo-FTC-alg-K-elem} satisfies $S_{\mathbf{F}} \circ \Lambda = \Lambda$, where $S_{\mathbf{F}}$ is the antipode of $\mathbf{F} := \mathbf{F}_{\mathcal{C}}$. Moreover, we have
  \begin{align*}
    (\id_{\mathbf{F}} \otimes m) \circ (\Delta \otimes \id_{\mathbf{F}}) \circ (\Lambda \otimes \Lambda)
    & = (\id_{\mathbf{F}} \otimes m) \circ (\id_{\mathbf{F}} \otimes \id_{\mathbf{F}} \otimes \Lambda) \circ \Delta \circ \Lambda \\
    & = (\id \otimes \varepsilon \otimes \Lambda) \circ \Delta \circ \Lambda \\
    & = \Lambda \otimes \Lambda.
  \end{align*}
  Hence $\Lambda$ is an {\em algebraic Kirby element} in the sense of Virelizier \cite[Definition~2.7]{MR2251160}. If $\Lambda$ is normalizable in his sense, then it gives rise to an invariant $\tau_{\mathcal{C}}(-;  \Lambda)$ of closed 3-manifolds. The invariant $\tau_{\mathcal{C}}(-; \Lambda)$ may be called the Hennings-Kauffman-Radford (HKR) invariant arising from $\mathcal{C}$, since the original HKR invariant \cite{MR1413901,MR1321293} constructed from a finite-dimensional unimodular ribbon Hopf algebra $H$ is the case where $\mathcal{C} = H\mbox{-{\sf mod}}$.
\end{remark}


\begin{thebibliography}{10}

\bibitem{MR2724388}
M.~Aguiar and S.~Mahajan.
\newblock {\em Monoidal functors, species and {H}opf algebras}, volume~29 of
  {\em CRM Monograph Series}.
\newblock American Mathematical Society, Providence, RI, 2010.
\newblock With forewords by Kenneth Brown and Stephen Chase and Andr{\'e}
  Joyal.

\bibitem{MR1797619}
B.~Bakalov and A.~Kirillov, Jr.
\newblock {\em Lectures on tensor categories and modular functors}, volume~21
  of {\em University Lecture Series}.
\newblock American Mathematical Society, Providence, RI, 2001.

\bibitem{MR1759389}
Y.~Bespalov, T.~Kerler, V.~Lyubashenko, and V.~Turaev.
\newblock Integrals for braided {H}opf algebras.
\newblock {\em J. Pure Appl. Algebra}, 148(2):113--164, 2000.

\bibitem{MR2793022}
A.~Brugui{\`e}res, S.~Lack, and A.~Virelizier.
\newblock Hopf monads on monoidal categories.
\newblock {\em Adv. Math.}, 227(2):745--800, 2011.

\bibitem{MR2355605}
A.~Brugui{\`e}res and A.~Virelizier.
\newblock Hopf monads.
\newblock {\em Adv. Math.}, 215(2):679--733, 2007.

\bibitem{MR2869176}
A.~Brugui{\`e}res and A.~Virelizier.
\newblock Quantum double of {H}opf monads and categorical centers.
\newblock {\em Trans. Amer. Math. Soc.}, 364(3):1225--1279, 2012.

\bibitem{MR1926102}
S.~Caenepeel, G.~Militaru, and S.~Zhu.
\newblock {\em Frobenius and separable functors for generalized module
  categories and nonlinear equations}, volume 1787 of {\em Lecture Notes in
  Mathematics}.
\newblock Springer-Verlag, Berlin, 2002.

\bibitem{MR3039775}
A.~Davydov, M.~M{\"u}ger, D.~Nikshych, and V.~Ostrik.
\newblock The {W}itt group of non-degenerate braided fusion categories.
\newblock {\em J. Reine Angew. Math.}, 677:135--177, 2013.

\bibitem{DayPhDThesis}
B.~Day.
\newblock {\em Construction of biclosed categories}.
\newblock PhD thesis, University of New South Wales, 1970.

\bibitem{MR2342829}
B.~Day and R.~Street.
\newblock Centres of monoidal categories of functors.
\newblock In {\em Categories in algebra, geometry and mathematical physics},
  volume 431 of {\em Contemp. Math.}, pages 187--202. Amer. Math. Soc.,
  Providence, RI, 2007.

\bibitem{MR1106898}
P.~Deligne.
\newblock Cat\'egories tannakiennes.
\newblock In {\em The {G}rothendieck {F}estschrift, {V}ol.\ {II}}, volume~87 of
  {\em Progr. Math.}, pages 111--195. Birkh\"auser Boston, Boston, MA, 1990.

\bibitem{2013arXiv1312.7188D}
C.~L. {Douglas}, C.~{Schommer-Pries}, and N.~{Snyder}.
\newblock {Dualizable tensor categories}.
\newblock {\tt arXiv:1312.7188}, 2013.

\bibitem{2014arXiv1406.4204D}
C.~L. {Douglas}, C.~{Schommer-Pries}, and N.~{Snyder}.
\newblock {The balanced tensor product of module categories}.
\newblock {\tt arXiv:1406.4204}, 2014.

\bibitem{MR0125148}
S.~Eilenberg.
\newblock Abstract description of some basic functors.
\newblock {\em J. Indian Math. Soc. (N.S.)}, 24:231--234 (1961), 1960.

\bibitem{MR0184984}
S.~Eilenberg and J.~C. Moore.
\newblock Adjoint functors and triples.
\newblock {\em Illinois J. Math.}, 9:381--398, 1965.

\bibitem{EGNO-Lect}
  P.~Etingof, S.~Gelaki, D.~Nikshych, and V.~Ostrik. Tensor categories. Lecture notes for MIT 18.769, 2009. \url{http://www-math.mit.edu/~etingof/tenscat1.pdf}.

\bibitem{MR2097289}
P.~Etingof, D.~Nikshych, and V.~Ostrik.
\newblock An analogue of {R}adford's {$S^4$} formula for finite tensor
  categories.
\newblock {\em Int. Math. Res. Not.}, (54):2915--2933, 2004.

\bibitem{MR2183279}
P.~Etingof, D.~Nikshych, and V.~Ostrik.
\newblock On fusion categories.
\newblock {\em Ann. of Math. (2)}, 162(2):581--642, 2005.

\bibitem{MR2119143}
P.~Etingof and V.~Ostrik.
\newblock Finite tensor categories.
\newblock {\em Mosc. Math. J.}, 4(3):627--654, 782--783, 2004.

\bibitem{LopezFranco2013207}
I.~L. Franco.
\newblock Tensor products of finitely cocomplete and abelian categories.
\newblock {\em Journal of Algebra}, 396(0):207 -- 219, 2013.

\bibitem{MR1413901}
M.~Hennings.
\newblock Invariants of links and {$3$}-manifolds obtained from {H}opf
  algebras.
\newblock {\em J. London Math. Soc. (2)}, 54(3):594--624, 1996.

\bibitem{MR2943680}
K.~Ishihara and A.~Ishii.
\newblock An operator invariant for handlebody-knots.
\newblock {\em Fund. Math.}, 217(3):233--247, 2012.

\bibitem{MR3265394}
A.~Ishii and A.~Masuoka.
\newblock Handlebody-knot invariants derived from unimodular {H}opf algebras.
\newblock {\em J. Knot Theory Ramifications}, 23(7):1460001, 24, 2014.

\bibitem{MR1321145}
C.~Kassel.
\newblock {\em Quantum groups}, volume 155 of {\em Graduate Texts in
  Mathematics}.
\newblock Springer-Verlag, New York, 1995.

\bibitem{MR1321293}
L.~H. Kauffman and D.~E. Radford.
\newblock Invariants of {$3$}-manifolds derived from finite-dimensional {H}opf
  algebras.
\newblock {\em J. Knot Theory Ramifications}, 4(1):131--162, 1995.

\bibitem{MR0360749}
G.~M. Kelly.
\newblock Doctrinal adjunction.
\newblock In {\em Category {S}eminar ({P}roc. {S}em., {S}ydney, 1972/1973)},
  pages 257--280. Lecture Notes in Math., Vol. 420. Springer, Berlin, 1974.

\bibitem{MR1862634}
T.~Kerler and V.~V. Lyubashenko.
\newblock {\em Non-semisimple topological quantum field theories for
  3-manifolds with corners}, volume 1765 of {\em Lecture Notes in Mathematics}.
\newblock Springer-Verlag, Berlin, 2001.

\bibitem{MR1712872}
S.~Mac~Lane.
\newblock {\em Categories for the working mathematician}, volume~5 of {\em
  Graduate Texts in Mathematics}.
\newblock Springer-Verlag, New York, second edition, 1998.

\bibitem{MR1243637}
S.~Montgomery.
\newblock {\em Hopf algebras and their actions on rings}, volume~82 of {\em
  CBMS Regional Conference Series in Mathematics}.
\newblock Published for the Conference Board of the Mathematical Sciences,
  Washington, DC, 1993.

\bibitem{MR2381536}
S.-H. Ng and P.~Schauenburg.
\newblock Higher {F}robenius-{S}chur indicators for pivotal categories.
\newblock In {\em Hopf algebras and generalizations}, volume 441 of {\em
  Contemp. Math.}, pages 63--90. Amer. Math. Soc., Providence, RI, 2007.

\bibitem{MR1976459}
V.~Ostrik.
\newblock Module categories, weak {H}opf algebras and modular invariants.
\newblock {\em Transform. Groups}, 8(2):177--206, 2003.

\bibitem{MR0407069}
D.~E. Radford.
\newblock The order of the antipode of a finite dimensional {H}opf algebra is
  finite.
\newblock {\em Amer. J. Math.}, 98(2):333--355, 1976.

\bibitem{2013arXiv1309.4539S}
K.~{Shimizu}.
\newblock {The pivotal cover and Frobenius-Schur indicators}.
\newblock To appear in Journal of Algebra.

\bibitem{2014arXiv1412.0211S}
K.~{Shimizu}.
\newblock {The relative modular object and Frobenius extensions of finite Hopf algebras}.
\newblock {\tt arXiv:1412.0211}, 2014.

\bibitem{MR1685417}
M.~Takeuchi.
\newblock Finite {H}opf algebras in braided tensor categories.
\newblock {\em J. Pure Appl. Algebra}, 138(1):59--82, 1999.

\bibitem{MR2251160}
A.~Virelizier.
\newblock Kirby elements and quantum invariants.
\newblock {\em Proc. London Math. Soc. (3)}, 93(2):474--514, 2006.

\bibitem{MR0118757}
C.~E. Watts.
\newblock Intrinsic characterizations of some additive functors.
\newblock {\em Proc. Amer. Math. Soc.}, 11:5--8, 1960.

\end{thebibliography}
\def\cprime{$'$}

\end{document}